\documentclass[11pt,leqno]{article}
\usepackage[doc]{optional}
\usepackage{enumitem}
\usepackage{pstricks}
\usepackage{pst-plot}
\usepackage{pst-3dplot}
\usepackage{pstricks-add}

\usepackage{color}
\usepackage{float}
\usepackage{soul}
\usepackage{graphicx}
\definecolor{labelkey}{rgb}{0,0.08,0.45}
\definecolor{refkey}{rgb}{0,0.6,0.0}

\definecolor{Brown}{rgb}{0.45,0.0,0.05}
\definecolor{lime}{rgb}{0.00,0.8,0.0}
\definecolor{lblue}{rgb}{0.5,0.5,0.99}
\definecolor{lblue}{rgb}{0.8,0.85,1.00}
\definecolor{anotherblue}{rgb}{.8, .8,1}
\definecolor{violet}{rgb}{0.9,0.6,0.9}
\definecolor{greenyellow}{rgb}{0.53,0.99,0.18}

\definecolor{Lyellow}{rgb}{0.87,0.87,0.87}
\definecolor{Lgray}{rgb}{0.92,0.92,0.92}
\definecolor{Mgray}{rgb}{0.5,0.5,0.5}
\definecolor{Gold}{rgb}{0.99,0.84,0.0}

\usepackage{mathpazo}
\usepackage{amsmath}
\usepackage{amssymb}
\usepackage{theorem}
\usepackage{fancyhdr}
\usepackage{empheq}

\usepackage[margin=1.0in]{geometry}

\usepackage{hyperref}

\newcommand{\nnn}{\ensuremath{{n\in{\mathbb N}}}}
\newcommand{\kkk}{\ensuremath{{k\in{\mathbb N}}}}
\newcommand{\thalb}{\ensuremath{\tfrac{1}{2}}}
\newcommand{\menge}[2]{\big\{{#1}~\big |~{#2}\big\}}
\newcommand{\mmenge}[2]{\bigg\{{#1}~\bigg |~{#2}\bigg\}}
\newcommand{\Menge}[2]{\left\{{#1}~\Big|~{#2}\right\}}

\newcommand{\fenv}[1]%
{\ensuremath{\,\overrightarrow{\operatorname{env}}_{#1}}}
\newcommand{\benv}[1]%
{\ensuremath{\,\overleftarrow{\operatorname{env}}_{#1}}}

\newcommand{\scal}[2]{\left\langle{#1},{#2}  \right\rangle}

\newcommand{\zeroun}{\ensuremath{\left]0,1\right[}}
\newcommand{\RR}{\ensuremath{\mathbb R}}
\newcommand{\NN}{\ensuremath{\mathbb N}}
\newcommand{\ZZ}{\ensuremath{\mathbb Z}}

\newcommand{\RP}{\ensuremath{\mathbb R}_+}
\newcommand{\RM}{\ensuremath{\mathbb R}_-}
\newcommand{\RPP}{\ensuremath{\mathbb R}_{++}}
\newcommand{\RMM}{\ensuremath{\mathbb R}_{--}}

\newcommand{\ball}{\ensuremath{\mathrm{ball}}}

\newcommand{\reli}{\ensuremath{\operatorname{ri}}}
\newcommand{\inte}{\ensuremath{\operatorname{int}}}

\newcommand{\bd}{\ensuremath{\operatorname{bdry}}}

\newcommand{\cone}{\ensuremath{\operatorname{cone}}}

\newcommand{\aff}{\ensuremath{\operatorname{aff}}}

\newcommand{\Id}{\ensuremath{\operatorname{Id}}}

\newcommand{\pinf}{\ensuremath{+\infty}}
\newcommand{\bx}{\ensuremath{\mathbf{x}}}

\newcommand{\blambda}{\ensuremath{{\boldsymbol{\lambda}}}}
\newcommand{\bmu}{\ensuremath{{\boldsymbol{\mu}}}}

\newcommand{\by}{\ensuremath{\mathbf{y}}}

\newcommand{\wt}[1]{\widetilde{#1}}
\newcommand{\nc}[2]{N^{#2}_{#1}}

\newcommand{\pn}[2]{\widehat{N}^{#2}_{#1}} 

\newcommand{\rate}{\ensuremath{\hat{\rho}}}

\def\ve{\varepsilon}
\def\dd{\delta}

\def\dn{\downarrow}
\def\rts{\kappa}
\def\rts{\hat{\kappa}}
\def\Rts{\eta}

{\begin{list}{}{%
\settowidth{\labelwidth}{\textrm{#1~}}%
\setlength{\leftmargin}{\labelwidth+\labelsep}}}
{\end{list}}
\newtheorem{theorem}{Theorem}[section]
\newtheorem{lemma}[theorem]{Lemma}
\newtheorem{corollary}[theorem]{Corollary}
\newtheorem{proposition}[theorem]{Proposition}
\newtheorem{definition}[theorem]{Definition}
\theoremstyle{plain}{\theorembodyfont{\rmfamily}
}
\theoremstyle{plain}{\theorembodyfont{\rmfamily}
}
\theoremstyle{plain}{\theorembodyfont{\rmfamily}
}
\theoremstyle{plain}{\theorembodyfont{\rmfamily}
\newtheorem{example}[theorem]{Example}}

\theoremstyle{plain}{\theorembodyfont{\rmfamily}
\newtheorem{remark}[theorem]{Remark}}

\newcommand{\boxedeqn}[1]{%
    \[\fbox{%
        \addtolength{\linewidth}{-2\fboxsep}%
        \addtolength{\linewidth}{-2\fboxrule}%
        \begin{minipage}{\linewidth}%
        \begin{equation}#1\\[+4mm]\end{equation}%
        \end{minipage}%
      }\]%
  }

\def\doi{DOI}
\newcommand{\papercited}[6]
{{\sc#1}, #2, \emph{#3} {\bf #4} (#5), #6.}

\newcommand{\others}[2]
{{\sc #1}, #2}

\newcommand{\bookcited}[4]
{{\sc #1}, \emph{#2}, #3 (#4).}

\newcounter{count}

\allowdisplaybreaks

\def\firstPic{\begin{figure}[h!]
\psset{xunit=1.2cm, yunit=1.2cm}
\begin{pspicture}
(-3,-2.5)(10,6.5)
\pscustom{
\psline[linestyle=none]{<->}(10,0)(8,0)(6,-0.4)(4,0)(3,-0.2)(2,0)(1.5,-0.1)(1,0)(0.75,-0.05)(0.5,0)(0,0)(-3,0)
\gsave
\psline(0,-2)(7,-2)
\fill[fillstyle=vlines,opacity=0.1,hatchcolor=Mgray,hatchwidth=0.1pt,hatchsep=10pt]
\grestore}
\psline[linewidth=0.4pt,linecolor=red,linestyle=dashed,showpoints=true,dotsize=3pt]{-}(8,0)(4,0)(2,0)(1,0)(0.5,0)(0.25,0)(0,0)
\pswedge[linestyle=none,fillstyle=solid,fillcolor=Gold](4,0){0.8}{76.5}{102.5}
\psline[linewidth=0.7pt,linecolor=red]{<->}(4.2,1)(4,0)(3.8,1)
\rput{-12}(4,0){\psline[linewidth=0.7pt](0.2,0)(0.2,0.2)(0,0.2)}
\rput{102}(4,0){\psline[linewidth=0.7pt](0.2,0)(0.2,0.2)(0,0.2)}
\psarc{->}(4,0){1.95}{-12}{0}
\rput(5.4,-0.15){\small$w$}
\rput{41}{
\pscustom{
\psline[linestyle=none]{<->}(10,0)(8,0)(6,0.4)(4,0)(3,0.2)(2,0)(1.5,0.1)(1,0)(0.75,0.05)(0.5,0)(0,0)(-3,0)
\gsave
\psline(0,2)(7,2)
\fill[fillstyle=hlines,opacity=0.1,hatchcolor=blue,hatchwidth=0.1pt,hatchsep=10pt]
\grestore}
\psline[linewidth=0.4pt,linecolor=blue,linestyle=dashed,showpoints=true,dotsize=3pt]{-}(8,0)(4,0)(2,0)(1,0)(0.5,0)(0.25,0)(0,0)
\pswedge[linestyle=none,fillstyle=solid,fillcolor=Gold](4,0){0.8}{-102.5}{-76.5}
\psline[linewidth=0.7pt,linecolor=red]{<->}(4.2,-1)(4,0)(3.8,-1)
}

\psline[linestyle=none]{->}(0,0)(-3,0)
\rput{21}
{\psline[linewidth=0.4pt,linecolor=gray,linestyle=dashed](0,0)(8,0)}
\rput(8,3){$y=(\tan2w) x$}
\rput(7,-1.5){$A$}
\rput(2,4){$B$}
\rput(8.3,-0.25){$s_0$}
\rput(-0.2,0.2){$c$}
\rput(5,1){$N_A(s_k)$}
\rput(4,2.5){$N_B(z_k)$}
\rput(4.2,-0.25){$s_k$}
\rput(2.2,-0.25){$s_{k+1}$}
\rput(2.8,2.8){$z_k$}
\rput(1.3,1.6){$z_{k+1}$}
\rput(6,5.5){$z_0$}
\end{pspicture}
\caption{Non-superregular sets in $\RR^2$}%
\label{pic:1}
\end{figure}} 

\def\secondPic{\begin{figure}[h!]
\psset{xunit=1cm, yunit=1cm}
\begin{pspicture}
(-3,-2.5)(10,6.5)
\rput{41}{
\pscustom{
\psline[linestyle=none]{->}(10,0.4)(8,0)(6,0.4)(4,0)(3,0.2)(2,0)(1.5,0.1)(1,0)(0.75,0.05)(0.5,0)(0,0)(-3,0)
\gsave
\psline(0,2)(7,2)
\fill[fillstyle=hlines,opacity=0.1,hatchcolor=blue,hatchwidth=0.1pt,hatchsep=10pt]
\grestore}
\psline[linewidth=0.4pt,linecolor=blue,linestyle=dashed,showpoints=true,dotsize=3pt]{-}(8,0)(4,0)(2,0)(1,0)(0.5,0)(0.25,0)(0,0)
}
\pscustom{
\psline[linestyle=none]{->}(10,-0.4)(8,0)(6,-0.4)(4,0)(3,-0.2)(2,0)(1.5,-0.1)(1,0)(0.75,-0.05)(0.5,0)(0,0)(-3,0)
\gsave
\psline(0,-2)(7,-2)
\fill[fillstyle=vlines,opacity=0.1,hatchcolor=Mgray,hatchwidth=0.1pt,hatchsep=10pt]
\grestore}
\psline[linewidth=0.4pt,linecolor=red,linestyle=dashed,showpoints=true,dotsize=3pt]{-}(8,0)(4,0)(2,0)(1,0)(0.5,0)(0.25,0)(0,0)
\pspolygon[linestyle=none,fillstyle=solid,opacity=0.6,fillcolor=Gold](4,0)(3.4,3.1)(4.25,4.25)(4.95,4.62)
\psline[linewidth=0.7pt,linecolor=red]{<->}(4.4,2)(4,0)(3.6,2)
\rput{-12}(4,0){\psline[linewidth=0.7pt](0.2,0)(0.2,0.2)(0,0.2)}
\rput{102}(4,0){\psline[linewidth=0.7pt](0.2,0)(0.2,0.2)(0,0.2)}
\psarc{->}(4,0){1.95}{-12}{0}
\rput(5.4,-0.15){\small$w$}
\rput{41}{\psline[linecolor=black]{-}(8,0)(6,0.4)(4,0)
\psline[linestyle=none]{->}(0,0)(-3,0)}
\rput(7,-1.5){$A$}
\rput(2,4){$B$}
\rput(8.3,-0.25){$s_{k-1}$}
\rput(-0.2,0.2){$c$}
\rput(4,2.5){$W$}
\rput(4.2,-0.25){$s_k$}
\rput(2.2,-0.25){$s_{k+1}$}
\rput(2.8,2.8){$z_k$}
\rput(1.3,1.6){$z_{k+1}$}
\rput(5.8,5.6){$z_{k-1}$}
\end{pspicture}
\caption{Inverse projections of $s_k$}
\label{pic:2}
\end{figure}}  

\def\thirdPic{\begin{figure}[h!]
\psset{xunit=1cm, yunit=1cm}
\begin{pspicture}
(-3,-2.5)(10,6.5)
\rput{41}{
\pscustom{
\psline[linestyle=none]{->}(10,0.4)(8,0)(6,0.4)(4,0)(3,0.2)(2,0)(1.5,0.1)(1,0)(0.75,0.05)(0.5,0)(0,0)(-3,0)
\gsave
\psline(0,2)(7,2)
\fill[fillstyle=hlines,opacity=0.1,hatchcolor=blue,hatchwidth=0.1pt,hatchsep=10pt]
\grestore}
\psline[linewidth=0.4pt,linecolor=blue,linestyle=dashed,showpoints=true,dotsize=3pt]{-}(8,0)(4,0)(2,0)(1,0)(0.5,0)(0.25,0)(0,0)}
\pscustom{
\psline[linestyle=none]{->}(10,-0.4)(8,0)(6,-0.4)(4,0)(3,-0.2)(2,0)(1.5,-0.1)(1,0)(0.75,-0.05)(0.5,0)(0,0)(-3,0)
\gsave
\psline(0,-2)(7,-2)
\fill[fillstyle=vlines,opacity=0.1,hatchcolor=Mgray,hatchwidth=0.1pt,hatchsep=10pt]
\grestore}
\psline[linewidth=0.4pt,linecolor=red,linestyle=dashed,showpoints=true,dotsize=3pt]{-}(8,0)(4,0)(2,0)(1,0)(0.5,0)(0.25,0)(0,0)
\psline[linewidth=0.4pt,linecolor=purple,linestyle=dashed,showpoints=false]{-}(8,0)(6.05,5.25)(4,0)
\psline[linewidth=0.4pt,linecolor=purple,linestyle=dashed,showpoints=false]{-}(6.05,5.25)(6,-0.4)
\psline[linewidth=0.6pt,linecolor=red,linestyle=dashed,showpoints=false]{-}(6.05,5.25)(7,-0.2)
\rput{41}{\psline[linestyle=none]{->}(0,0)(-3,0)}
\rput(7,-1.5){$A$}
\rput(2,4){$B$}
\rput(8.3,-0.25){$s_{k}$}
\rput(-0.2,0.2){$c$}
\rput(4.2,-0.25){$s_{k+1}$}
\rput(2.7,2.8){$z_{k+1}$}
\rput(5.8,5.6){$z_k$}
\rput(6,-0.6){$s$}
\rput(7,-0.45){$h$}
\rput(5,-0.45){$h'$}
\psdots[dotsize=2.5pt](7,-0.2)(5,-0.2)
\end{pspicture}
\caption{Projections of $z_k$ on $A$}\label{pic:3}
\end{figure}} 

\def\MMARP{\begin{figure}[h!]
\begin{center}
\psset{xunit=1.4cm, yunit=1.4cm}
\begin{pspicture}
(-3,-0.7)(6,1.7)
\psline[linewidth=1pt,linestyle=solid,showpoints=true,dotstyle=*,dotsize=4pt]{-}(-3,0)(-2.5,0)(-2,0)(-1,0)(0,0)(1,0)(2,0)(6,0)
\psellipticarc[arrowsize=.2,linestyle=dashed,linewidth=0.7pt,linecolor=red]{<-}(1,-2)(2.24,2.24){65}{115}
\psellipticarc[arrowsize=.2,linestyle=dashed,linewidth=0.7pt,linecolor=red]{<-}(4,-4.6)(5.02,5.02){69}{111}
\psellipticarc[arrowsize=.2,linestyle=dashed,linewidth=0.7pt,linecolor=red]{->}(4,-1.7)(2.62,2.62){43}{137}
\psellipticarc[arrowsize=.2,linestyle=dashed,linewidth=0.7pt,linecolor=red]{<-}(4,-0.72)(2.13,2.13){23}{157}
\rput(4,1.05){$\cdots$}
\rput(-3,0.3){$-3$}
\rput(-2,0.3){$-2$}
\rput(-1,0.3){$-1$}
\rput(0,0.3){$0$}
\rput(1.9,0.3){$2$}
\rput(6.1,0.3){$6$}
\rput(4,1.7){\red{MAP}}
\rput(-1,-0.8){\blue{$\tfrac{1}{2}$-MARP}}
\rput(-2.75,-0.3){$\cdots$}
\psellipticarc[arrowsize=.2,linestyle=dashed,linewidth=0.7pt,linecolor=blue]{->}(0.5,0.375)(0.625,0.625){-140}{-40}
\psellipticarc[arrowsize=.2,linestyle=dashed,linewidth=0.7pt,linecolor=blue]{<-}(0,0.75)(1.25,1.25){-140}{-40}
\psellipticarc[arrowsize=.2,linestyle=dashed,linewidth=0.7pt,linecolor=blue]{<-}(-1.5,0.182)(0.532,0.532){-155}{-25}
\psellipticarc[arrowsize=.2,linestyle=dashed,linewidth=0.7pt,linecolor=blue]{<-}(-2.25,0.06)(0.26,0.26){-160}{-20}
\end{pspicture}
\end{center}
\caption{MAP vs MARP}
\label{pic:4}
\end{figure}}

\begin{document}

\title{\textrm{ The Method of Alternating Relaxed Projections\\
for two nonconvex sets}}

\author{
Heinz H.\ Bauschke\thanks{Mathematics, University of British
Columbia, Kelowna, B.C.\ V1V~1V7, Canada. E-mail:
\texttt{heinz.bauschke@ubc.ca}.}, Hung M.\
Phan\thanks{Mathematics and Statistics, University of Victoria, PO Box 3060 STN CSC, Victoria, B.C.\ V8W~3R4, Canada. E-mail:  \texttt{hphan@uvic.ca}.}, ~and
Xianfu\ Wang\thanks{Mathematics, University of British Columbia,
Kelowna, B.C.\ V1V~1V7, Canada. E-mail:
\texttt{shawn.wang@ubc.ca}.}}

\date{May 18, 2013}

\maketitle \thispagestyle{fancy}

\bigskip

\begin{center}
\em Dedicated to Boris Mordukhovich on the occasion of his 65th Birthday
\end{center}

\bigskip

\begin{abstract} \noindent
The Method of Alternating Projections (MAP),
a classical algorithm for solving
feasibility problems, has recently been intensely studied
for nonconvex sets.
However, intrinsically available are only local convergence results:
convergence occurs if the starting point is not too far away from solutions to avoid
getting trapped in certain regions.
Instead of taking full projection steps, it can be advantageous to
underrelax, i.e., to move only part way towards the constraint
set, in order to enlarge the regions of convergence.

In this paper, we thus systematically study the
Method of Alternating Relaxed Projections (MARP)
for two (possibly nonconvex) sets.
Complementing our recent work on MAP,
we establish local linear convergence
results for the MARP.
Several examples illustrate our analysis.
\end{abstract}

{\small \noindent {\bfseries 2010 Mathematics Subject
Classification:} 
Primary 65K05;
Secondary 47N10, 49J52, 49M20, 65K10, 90C25, 90C26.

}

\section{Introduction}

\label{sec:intro}

We assume throughout this paper that
\boxedeqn{
\text{$X$ is a Euclidean space}
}
with inner product $\scal{\cdot}{\cdot}$ and associated norm
$\|\cdot\|$ and that
\boxedeqn{
\label{e:AandB}
\text{$A$ and $B$ are nonempty closed subsets of $X$.}
}
Our aim is to solve the feasibility problem
\begin{equation}
\text{find $x\in A\cap B$.}
\end{equation}
(We do not \emph{a priori} assume that $A\cap B\neq\varnothing$.)
We assume that it is possible to evaluate the \emph{projection
operators} (nearest point mappings) $P_A$ and $P_B$ associated with the constraints sets
$A$ and $B$ respectively. The operators $P_A$ and $P_B$ are
generally set-valued; they are single-valued only in the convex
case. The celebrated \emph{Method of Alternating Projections
(MAP)}, whose origins can be traced back to von Neumann \cite{vN}
and Wiener \cite{Wiener},
with starting point $b_{-1}\in X$ generates sequences
according to the update rule\footnote{We follow a common but
convenient abuse of notation and write $a_n=P_Ab_{n-1}$ etc.\ if
the set of nearest points is a singleton.}
\begin{equation}
(\forall\nnn)\quad
a_{n}\in P_Ab_{n-1}
\quad\text{and}\quad
b_{n} \in P_Ba_n.
\end{equation}
If $A$ and $B$ are convex, then this method is well understood;
see, e.g.,
\cite{BBR,BBvN,BC2011,BR,Cegielski,CenZen,Deutsch,DHp1,DHp2,DHp3}
and the references therein for extensions and variants.
The convergence theory for the MAP and related methods
is much more delicate in the absence of convexity;
see, e.g., \cite{CombT,LLM,Luke,zero,one} and
the references therein.

Simple examples can be constructed to show that in general one
cannot expect global convergence of the MAP when $A\cap
B\neq\varnothing$:

\begin{example}[unrelaxed MAP]
\label{ex:130430a}
Suppose that $X=\RR$, that
$A=\{-3,2\}$ and that $B=\{-3,6\}$.
Then $A\cap B=\{-3\}\neq\varnothing$.
Now set $b_{-1}:=0$.
Then $a_0 := P_Ab_{-1}=P_A0 = 2$ (since $|2-0|=2<3=|-3-0|$)
and $b_0 := P_Ba_0 = P_B2 = 6$
(since $|6-2|=4<5=|-3-2|$) and clearly $a_1 := P_Ab_0 = 2$.
It follows that
\begin{equation}
\text{$(\forall\nnn)$ $a_n=2$ and $b_n=6$.}
\end{equation}
Thus, the sequences generated by the MAP do not converge to
a point in $A\cap B$ (see Figure~\ref{pic:4}).
\end{example}

To improve this situation, we study in this paper the
\emph{Method of Alternating Relaxed Projections}, where
the unrelaxed projection steps are replaced by underrelaxed
versions; e.g., the projection operators $P_A$ and $P_B$ may be
replaced by $(1-\lambda)\Id+\lambda P_A$ and
$(1-\mu)\Id+\mu P_B$, where $\lambda$ and $\mu$ belong
to $\left]0,1\right]$.
In the convex case, there are several pertinent references
including \cite{BBR,BCC,BR,Comb04,GK,GR,GPR,WangBau}.

The idea of \emph{regularizing} operators is of course not new;
MARP can be seen as regularizing the straight projection
operators. To demonstrate the potential of this approach,
let us revisit Example~\ref{ex:130430a}:

\begin{example}[MARP for Example~\ref{ex:130430a}]
\label{ex:130430b}
Let $X$, $A$, $B$, and $b_{-1}$ be as in
Example~\ref{ex:130430a}.
Rather than iterating $P_A$ and $P_B$, we now iterate
$\thalb\Id + \thalb P_A$ and $\thalb\Id +\thalb P_B$.
Then
$a_0 = (\thalb\Id + \thalb P_A)b_{-1} = \thalb b_{-1}+\thalb P_Ab_{-1}
=\thalb 0 + \thalb 2 = 1$,
$b_0 = ( \thalb\Id + \thalb P_B)a_0 = \thalb 1 + \thalb (-3) =
-2$,
$a_1 = (\thalb\Id + \thalb P_A)b_{0} =
\thalb(-2)+\thalb(-3)=-5/2$, $b_1 = \cdots = -11/4$,
$a_2 = -23/8$, $b_2 = -47/16$, \ldots
(see Figure~\ref{pic:4}),
and the sequences generated converge\footnote{This will follow
from Example~\ref{ex:130430c} below.} to $-3$, the unique
point in $A\cap B$, as desired.
\end{example}

\MMARP

{\em The goal of this paper is to systematically study the MARP
and to provide conditions sufficient for convergence.}

The tools used are from variational analysis; we extend
techniques recently introduced in \cite{zero,one}.

Our main results are the following:
\begin{itemize}
\item Theorem~\ref{t:geo} is a powerful abstract linear
convergence result that is applicable in particular to the MARP;
\item Theorem~\ref{t:loc-viaCQ} provides a local linear
convergence result for the MARP in the presence of a CQ
condition;
\item Theorem~\ref{t:main2} guarantees local linear convergence
of the MARP under some regularity assumptions.
\end{itemize}

The paper is organized as follows:
After reviewing auxiliary notions in Section~\ref{sec:aux},
we introduce the MARP in Section~\ref{sec:marp} and obtain
some basic properties. Abstract linear convergence results are
presented in Section~\ref{sec:abstract}.
Local linear convergence results based on CQ conditions and on
regularity are
provided in Sections~\ref{sec:CQ} and \ref{sec:regu},
respectively.
In Section~\ref{sec:vanish}, we discuss linearly vanishing
relaxation parameters.
Various examples illustrating the general theory are constructed
in Sections~\ref{sec:examples} and \ref{sec:doubly}.

We conclude this section with some notational comments.
We write $\RP=\menge{x\in\RR}{x\geq 0}$,
$\ZZ = \{0,\pm 1,\pm 2,\ldots\}$, and $\NN=\ZZ\cap\RP$.
The \emph{distance function} is $d_A\colon
x\mapsto \inf_{a\in A}\|x-a\|$ and
the (generally set-valued) \emph{projection operator}
is $P_A\colon x\mapsto\menge{a\in A}{\|x-a\|=d_A(x)}$.
Given a subset $S$ of $X$, we write
$\inte S$, $\reli S$, and $\overline{S}$
for the interior, the relative interior and the closure of $S$,
respectively.
If $u$ and $v$ are points in $X$, we write
$[u,v] = \menge{(1-\lambda)u+\lambda v}{0\leq \lambda \leq 1}$,
$\left]u,v\right] =  \menge{(1-\lambda)u+\lambda v}{0<\lambda
\leq 1}$, and similarly for
$\left[u,v\right[$ and $\left]u,v\right[$.
For notation not explicitly stated in this paper, and background
material in convex and variational analysis,
we refer the reader to \cite{BC2011,BorVanBook,Boris1,Rock70,Rock98,Zalinescu}.

\section{Auxiliary Notions}

\label{sec:aux}

In this section, we collect several technical definitions
for future use. For further results and comments,
see \cite{zero,one} and the references therein.

\begin{definition}[restricted normal cones]
{\rm (See \cite[Definition~2.1]{zero}.)}
Let $a\in A$.
\begin{enumerate}
\item
\label{d:pnB} The \emph{$B$-restricted proximal normal cone} of $A$
at $a$ is
\begin{equation}\label{e:pnB}
\pn{A}{B}(a):= \cone\Big(\big(B\cap P_A^{-1}a\big)-a\Big) =
\cone\Big(\big(B-a\big)\cap \big(P_A^{-1}a-a\big)\Big).
\end{equation}
\item\label{d:nc}
The \emph{$B$-restricted normal cone} $\nc{A}{B}(a)$ is implicitly
defined by $u\in\nc{A}{B}(a)$ if and only if there exist sequences
$(a_n)_\nnn$ in $A$ and $(u_n)_\nnn$ in $\pn{A}{B}(a_n)$ such that
$a_n\to a$ and $u_n\to u$.
\end{enumerate}
\end{definition}

\begin{definition}[regularity of sets]
{\rm (See \cite[Definition~8.1]{zero}.)}
\label{d:reg}
Let $c\in B$, $\ve\geq 0$, and $\dd>0$.
Then $B$ is \emph{$(A,\ve,\dd)$-regular at $c$} if
\begin{equation}
\label{e:dereg} \left.
\begin{array}{c}
(y,b)\in B\times B,\\
\|y-c\| \leq \dd,\|b-c\|\leq \dd,\\
u\in \pn{B}{A}(b)
\end{array}
\right\}\quad\Rightarrow\quad \scal{u}{y-b}\leq
\ve\|u\|\cdot\|y-b\|.
\end{equation}
The set $B$ is called $A$-\emph{superregular} at $c\in B$ if for
every $\ve>0$ there exists $\dd>0$ such that $B$ is
$(A,\ve,\dd)$-regular at $c$. When $A=X$, we say
``\emph{$B$ is $(\ve,\dd)$-regular}'' or
``\emph{$B$ is superregular}'', i.e., the prefix
``$X$-'' is omitted.
\end{definition}



\begin{definition}[linear convergence]
\label{d:linCon}
Let $(x_n)_\nnn$ be a sequence in $X$, let $c\in X$,
let $\alpha\in\zeroun$.
Then
\emph{$(x_n)_\nnn$ converges to $c$ linearly with rate $\alpha$}
if there exists $M\in\RP$ such
that\footnote{Note that one may alternatively
require that \eqref{e:130423a} only holds eventually
at the expense of possibly enlarging $M$; see, e.g., \cite[Remark~3.7]{one}.}
\begin{equation}
\label{e:130423a}
(\forall\nnn)\quad \|x_n-c\|\leq M\alpha^n.
\end{equation}
\end{definition}

\begin{definition}[CQ-number]
{\rm (See \cite[Definition~6.1]{zero}.)}
\label{d:CQn}
Let $\wt{A}$ and $\wt{B}$ be nonempty subsets of $X$,
let $c\in X$, and let $\dd\in\RPP$.
The \emph{CQ-number} at $c$ associated with
$(A,\wt{A},B,\wt{B})$ and $\dd$ is
\begin{equation}
\label{e:CQn}
\theta_\dd:=\theta_\dd\big(A,\wt{A},B,\wt{B}\big)
:=\sup\mmenge{\scal{u}{v}}
{\begin{aligned}
&u\in\pn{A}{\wt{B}}(a),v\in-\pn{B}{\wt{A}}(b),\|u\|\leq 1, \|v\|\leq 1,\\
&\|a-c\|\leq\dd,\|b-c\|\leq\dd.
\end{aligned}},
\end{equation}
and the \emph{limiting CQ-number} at $c$ associated with
$(A,\wt{A},B,\wt{B})$ is
\begin{equation}\label{e:lCQn}
\overline\theta:=\overline\theta\big(A,\wt{A},B,\wt{B}\big)
:=\lim_{\dd\dn0}\theta_\dd\big(A,\wt{A},B,\wt{B}\big).
\end{equation}
\end{definition}

\begin{definition}[CQ condition]
{\rm (See \cite[Definition~6.6]{zero}.)}
\label{d:CQ}
Let $\wt{A}$ and $\wt{B}$ be nonempty subsets of $X$,
and let $c\in X$.
Then the \emph{$(A,\wt{A},B,\wt{B})$-CQ condition} holds at $c$ if
\begin{equation}\label{e:CQ}
\nc{A}{\wt{B}}(c) \cap\big(-\nc{B}{\wt{A}}(c)\big)\subseteq\{0\}.
\end{equation}
\end{definition}
We recall the following equivalence from \cite[Theorem~6.8]{zero}:
\begin{equation}\label{e:CQ-CQn}
\nc{A}{\wt{B}}(c) \cap\big(-\nc{B}{\wt{A}}(c)\big)\subseteq\{0\}\ \Leftrightarrow\ \overline\theta(A,\wt{A},B,\wt{B})<1.
\end{equation}

\section{MARP: Basic Properties}

\label{sec:marp}

\begin{definition}
Let $y\in X$ and let $\lambda\in\left]0,1\right]$.
Then the vectors in the set
\begin{equation}
(1-\lambda)y + \lambda P_Ay = \menge{(1-\lambda)y+\lambda a}{a\in P_Ay}
\end{equation}
are called $\lambda$-relaxed projections of $y$ on $A$.
\end{definition}
Note that the $1$-relaxed projections are precisely the original
(unrelaxed) projections.
From now on, we assume that
\boxedeqn{
\label{e:blbm}
\text{$\blambda = (\lambda_n)_\nnn$ and $\bmu = (\mu_n)_\nnn$ are sequences in
$\left]0,1\right]$, \; and $\alpha_0:=\max\{\lambda_0,\mu_0\}$.}
}

\begin{definition}[Method of Alternating Relaxed Projections (MARP)]
\label{d:MARP}
Let $y_{-1}\in X$ be the starting point.
The method of alternating $(\blambda,\bmu)$-relaxed projections between $A$ and $B$
(the $(\blambda,\bmu)$-MARP or just MARP in short)
generates sequences $\bx := (x_n)_\nnn$
and $\by := (y_n)_\nnn$ as follows:
\begin{equation}
\begin{aligned}
(\forall\nnn)\quad y_{n-1}&\ \mapsto\ a_n\in P_A y_{n-1}
\ \mapsto\ x_n:=(1-\lambda_n)y_{n-1}+\lambda_n a_n\\
&\ \mapsto\ b_n\in P_B x_n
\ \mapsto\ y_n:=(1-\mu_n)x_{n}+\mu_n b_n \mapsto \cdots.
\end{aligned}
\end{equation}
We call $(\bx,\by)$ also $(\blambda,\bmu)$-MARP or simply MARP sequences.
\end{definition}
When $(\forall\nnn)$ $\lambda_n=\mu_n=1$,
then $(x_n)_\nnn=(a_n)_\nnn$ and $(y_n)_\nnn= (b_n)_\nnn$,
and the MARP reduces to the classic method of alternating
projections (MAP).

Unless specified otherwise,
we assume for the remainder of this paper that
\boxedeqn{
\label{e:bxby}
\text{$\big(\bx = (x_n)_\nnn,
\by = (y_n)_\nnn\big)$ are $(\blambda,\bmu)$-MARP sequences with
starting point $y_{-1}$.}
}

The following simple result turns out to be quite useful.

\begin{proposition}\label{p:01}
Let $y\in X$, $a\in P_A y$, $\lambda\in\left]0,1\right]$,
and set $x:=(1-\lambda)y+\lambda a$.
Then the following hold:
\begin{enumerate}
\item\label{p:01i} $P_A(x)=a$.
\item\label{p:01ii} $x-y=\lambda(a-y)$ and thus $\|x-y\|=\lambda\|a-y\|=\lambda d_A(y)$.
\item\label{p:01iii} $\lambda(x-a)=(1-\lambda)(y-x)$.
\end{enumerate}
\end{proposition}
\begin{proof}
\ref{p:01i}:
Suppose that $a'\in A\smallsetminus\{a\}$.
\emph{Case~1:} $x\in [y,a']$.
Then $\|y-a'\|>\|y-a\|$ because $y,a,a'$ lie on the same ray.
So
\begin{equation}
\|x-a'\|=\|y-a'\|-\|y-x\|>\|y-a\|-\|y-x\|=\|x-a\|.
\end{equation}
\emph{Case~2:} $x\notin [y,a']$. Then
\begin{equation}
\|x-a'\|>\|y-a'\|-\|y-x\|\geq\|y-a\|-\|y-x\|=\|x-a\|.
\end{equation}
In either case, $\|x-a'\|>\|x-a\|$ and therefore
$a=P_A(x)$.

\ref{p:01ii}: Indeed, $x-y=\lambda(a-y)$ $\Leftrightarrow$
$x=(1-\lambda)y+\lambda a$.

\ref{p:01iii}: We have:
$\lambda(x-a)=(1-\lambda)(y-x)$ $\Leftrightarrow$
$\lambda(x-a)=(1-\lambda)y-(1-\lambda)x$ $\Leftrightarrow$
$-\lambda a =(1-\lambda)y - x$ $\Leftrightarrow$
$x=(1-\lambda)y+\lambda a$.
\end{proof}


\begin{definition}[projection absorbing]\label{d:prjab}
Let $S$ be a nonempty subset of $X$.
Then $S$ is $A$-{\em projection absorbing} (or
\emph{projection absorbing with respect to $A$}), if
\begin{equation}\label{e:prjab}
(\forall s\in S)(\forall a\in P_A s)\quad [s,a]\subseteq S.
\end{equation}
\end{definition}

\begin{remark}
Let $S$ be a subset of $X$ that is $A$-projection absorbing.
\begin{enumerate}
\item
Clearly, $X$ is $A$-projection absorbing.
\item If $S$ is $B$-projection absorbing, then $S$ is also
$A\cup B$-projection absorbing because
$(\forall s\in S)$ $P_{A\cup B}(s)\subseteq P_{A}s \cup P_Bs$.
The opposite implication is not necessarily true:
for example, if $X=\RR^2$,
$S=A=\RR\times\{1\}$,
and $B=\RR\times \{0\}$, then
$S$ is $A\cup B$-projection absorbing but not $B$-projection absorbing.
\item If $S$ is convex and $P_A(S)\subseteq S$, then $S$ is $A$-projection absorbing.
\item On the other hand,
if $A = \ball(0;1)$ and $S = \overline{X\smallsetminus A}$, then
$S$ is not convex but still $A$-projection absorbing.
\end{enumerate}
\end{remark}

The notion of a projection absorbing set is important because of the
following result pertaining to the orbit of the MARP.

\begin{proposition}\label{p:1216a}
Let $S$ be a subset of $X$ that is both
$A$-projection absorbing and $B$-projection absorbing.
If $y_{-1}\in S$, then $(x_n)_\nnn$ and $(y_n)_\nnn$ lie in $S$.
\end{proposition}
\begin{proof}
This follows readily by using mathematical induction.
\end{proof}

\begin{lemma}
\label{l:first}
Set $\beta:=\max\{d_A(y_{-1}),d_B(y_{-1})\}$.
Then the following hold:
\begin{subequations}\label{e:tMARP1}
\begin{align}
\|x_0-y_{-1}\|&=\lambda_0 d_A(y_{-1})\leq\lambda_0\beta,\label{e:130425a}\\
\max\{d_A(x_0),d_B(x_0)\}&\leq\|x_0-y_{-1}\|+\beta\leq(1+\lambda_0)\beta,\label{e:130425b}\\
\|y_0-x_0\|&=\mu_0 d_B(x_0)\leq\mu_0(1+\lambda_0)\beta, \label{e:130425c}\\
\max\{\|y_0-x_0\|,\|x_0-y_{-1}\|\}&\leq \alpha_0(1+\alpha_0)\beta.  \label{e:130425d}
\end{align}
\end{subequations}
\end{lemma}
\begin{proof}
Using Proposition~\ref{p:01}\ref{p:01ii}, we have
$\|x_0-y_{-1}\| = \lambda_0d_A(y_{-1}) \leq \lambda_0\beta$.
Thus, \eqref{e:130425a} holds.
The nonexpansiveness of distance functions implies \eqref{e:130425b}.
On the one hand, using Proposition~\ref{p:01}\ref{p:01ii} again, we see that
$\|y_0-x_0\|=\mu_0 d_B(x_0)$.
On the other hand, \eqref{e:130425b} yields
$d_B(x_0)\leq (1+\lambda_0)\beta$.
Altogether, we obtain \eqref{e:130425c}.
Finally, \eqref{e:130425d} follows from \eqref{e:130425a} and
\eqref{e:130425c}.
\end{proof}

The following lemma is important for our analysis.

\begin{lemma}\label{l:MARP}
Let $\nnn$ and
let $\theta\in\left[0,1\right]$ be such that
\begin{equation}\label{e:l-ef01}
\scal{y_n-x_n}{y_{n-1}-x_n}\leq \theta\|y_n-x_n\|\cdot\|x_n-y_{n-1}\|.
\end{equation}
Then
\begin{equation}\label{e:l-ef02}
\|x_{n+1}-y_n\|\leq
\tfrac{\lambda_{n+1}}{\lambda_n}\sqrt{\lambda_n^2+(1-\lambda_n)^2+2\theta\lambda_n(1-\lambda_n)}
\cdot\max\big\{\|y_n-x_n\|,\|x_n-y_{n-1}\|\big\}.
\end{equation}
\end{lemma}
\begin{proof}
Proposition~\ref{p:01}\ref{p:01iii} yields
\begin{equation}\label{e:l-ef4}
x_n-a_n=\tfrac{1-\lambda_n}{\lambda_n}(y_{n-1}-x_n).
\end{equation}
Combining \eqref{e:l-ef4} with assumption \eqref{e:l-ef01} we have
\begin{equation}\label{e:l-ef5}
\scal{y_n-x_n}{x_n-a_n}=\tfrac{1-\lambda_n}{\lambda_n}\scal{y_n-x_n}{y_{n-1}-x_n}
\leq \tfrac{\theta(1-\lambda_n)}{\lambda_n}\|y_n-x_n\|\cdot\|x_n-y_{n-1}\|.
\end{equation}
Substituting \eqref{e:l-ef4} and \eqref{e:l-ef5} into
\begin{equation}\label{e:l-ef2}
\|y_n-a_n\|^2=\|y_n-x_n\|^2+\|x_n-a_n\|^2 + 2\scal{y_n-x_n}{x_n-a_n}
\end{equation}
gives
\begin{equation}\label{e:l-ef3}
\|y_n-a_n\|^2
\leq\|y_n-x_n\|^2+\tfrac{(1-\lambda_n)^2}{\lambda_n^2}\|x_n-y_{n-1}\|^2
+ 2\tfrac{\theta(1-\lambda_n)}{\lambda_n}\|y_n-x_n\|\cdot\|x_n-y_{n-1}\|.
\end{equation}
Multiplying both sides by $\lambda_{n+1}^2$, we have
\begin{equation}\label{e:l-ef3s}
\lambda_{n+1}^2\|y_n-a_n\|^2
\leq\tfrac{\lambda_{n+1}^2}{\lambda_n^2}
\big(\lambda_n^2+(1-\lambda_n)^2+2\theta\lambda_n(1-\lambda_n)\big)
\max{\negthinspace}^2\big\{\|y_n-x_n\|,\|x_n-y_{n-1}\|\big\}.
\end{equation}
From Proposition~\ref{p:01}\ref{p:01ii}, we have
\begin{equation}
\|x_{n+1}-y_n\|=\lambda_{n+1} d_A(y_n)\|\leq\lambda_{n+1}\|y_n-a_n\|.
\end{equation}
Combining with \eqref{e:l-ef3s}, we obtain the result.
\end{proof}


A proof analogous to that of Lemma~\ref{l:MARP}
(or interchanging the roles of $A$ and $B$)
yields the following result.

\begin{lemma}\label{l:MARP'}
Let $\nnn$ and
let $\theta\in\left[0,1\right]$ be such that
\begin{equation}\label{e:l-ef01'}
\scal{x_{n+1}-y_n}{x_{n}-y_n}\leq \theta\|x_{n+1}-y_n\|\cdot\|y_n-x_{n}\|.
\end{equation}
Then
\begin{equation}\label{e:l-ef02'}
\|y_{n+1}-x_{n+1}\|\leq
\tfrac{\mu_{n+1}}{\mu_n}\sqrt{\mu_n^2+(1-\mu_n)^2+2\theta\mu_n(1-\mu_n)}
\cdot\max\big\{\|x_{n+1}-y_n\|,\|y_n-x_{n}\|\big\}.
\end{equation}
\end{lemma}

\section{Abstract Linear Convergence}

\label{sec:abstract}

In this section, we provide convergence results that refine and complement
those of
\cite[Proposition~3.8]{one} and \cite{LLM}\footnote{In fact, the results in
this section hold true in any complete metric space.}.

\begin{lemma}[abstract linear convergence]
\label{l:geo}
Let $(x_n)_\nnn$ and $(y_n)_{n\geq -1}$ be sequences in $X$.
Assume that there exist constants $M\in\RP$ and $\rho\in\left[0,1\right[$ such that
\begin{equation}
(\forall\nnn)\quad\max\big\{d(y_n,x_n),d(x_n,y_{n-1})\big\}\leq
M \rho^n.
\end{equation}
Then there exists $\bar{c}\in X$ such that
\begin{equation}
(\forall\nnn)\quad\max\big\{d(x_n,\bar{c}),d(y_n,\bar{c})\big\}
\leq \frac{M(1+\rho)}{1-\rho}\cdot \rho^n;
\end{equation}
consequently, $(x_n)_\nnn$ and $(y_n)_\nnn$
converge linearly to $\bar c$ with rate $\rho$.
\end{lemma}
\begin{proof}
We have
\begin{equation}
(\forall\nnn)\quad d(y_{n},y_{n-1})\leq
d(y_{n},x_n)+d(x_n,y_{n-1})\leq2\rho^{n}M.
\end{equation}
Hence, for every $k\in\{n+1,n+2,\ldots\}$,
\begin{equation}
\label{e:130425e}
  d(y_k,y_n)\leq\sum_{i=n+1}^k d(y_i,y_{i-1})\leq 2M\sum_{i=n+1}^k
  \rho^{i}\leq \frac{2M\rho^{n+1}}{1-\rho}.
\end{equation}
Thus $(y_n)_\nnn$ is a Cauchy sequence with, say, limit $\bar c\in X$.
Letting $k\to+\infty$ in \eqref{e:130425e}, we see that
\begin{equation}
  d(y_n,\bar c)\leq\frac{2M\rho^{n+1}}{1-\rho}.
\end{equation}
It also follows that
\begin{equation}
\label{e:130430a}
d(x_n,\bar c)\leq d(x_n,y_n)+d(y_n,\bar c)\leq
M\rho^n+\frac{2M\rho^{n+1}}{1-\rho}=
\frac{M(1+\rho)\rho^n}{1-\rho}.
\end{equation}
Therefore, \eqref{e:130430a} implies that
\begin{equation}
(\forall\nnn)\quad\max\big\{d(x_n,\bar c),d(y_n,\bar c)\big\}\leq
\frac{M(1+\rho)}{1-\rho}\cdot \rho^n,
\end{equation}
as claimed.
\end{proof}


\begin{definition}[alternating contraction property]\label{d:contr}
Let $(x_n)_\nnn$ and $(y_n)_{n\geq -1}$ be sequences in $X$,
let $c\in X$, and let $(r,\rho)\in\RPP\times\left[0,1\right[$.
We say that $(x_n)_\nnn$ and $(y_n)_{n\geq-1}$ have
the \emph{alternating contraction property at $c$ with parameters
$(r,\rho)$} if the following implication holds whenever
$\nnn$ and
$(u_1,u_2,u_3,u_4)\in\{(y_{n-1},x_{n},y_n,x_{n+1}),(x_n,y_n,x_{n+1},y_{n+1})\}$:
\begin{equation}\label{e:contr}
\max\big\{d(u_2,c),d(u_3,c)\big\}\leq r
\quad\Rightarrow\quad d(u_3,u_4)\leq\rho
\max\big\{d(u_1,u_2),d(u_2,u_3)\big\}.
\end{equation}
\end{definition}

\begin{theorem}[abstract linear convergence]\label{t:geo}
Let $(x_n)_\nnn$ and $(y_n)_{n\geq -1}$ be sequences in $X$,
and let $(r,\rho)\in\RPP\times \left[0,1\right[$.
Assume that the sequences $(x_n)_\nnn$ and $(y_n)_{n\geq-1}$ have
the alternating contraction property at $y_{-1}$ with parameters
$(r,\rho)$.
Assume further that
\begin{equation}\label{e:geo1}
  M:=\max\big\{d(y_0,x_0),d(x_0,y_{-1})\big\}\leq\tfrac{ r(1-\rho)}{2}.
\end{equation}
Then $(x_n)_\nnn$ and $(y_n)_\nnn$ converge linearly
to a point $\bar c\in X$ with rate $\rho$;
more precisely,
\begin{equation}
(\forall\nnn)\quad
\max\big\{d(x_n,\bar c),d(y_n,\bar c)\big\}\leq
\tfrac{M(1+\rho)}{1-\rho} \rho^n
\leq\tfrac{ r(1+\rho)}{2} \rho^n.
\end{equation}
In addition,
$(x_n)_\nnn$ and $(y_n)_\nnn$, and hence $\bar{c}$,
all lie in $\ball(y_{-1};r)$.
\end{theorem}
\begin{proof}
Using \eqref{e:geo1}, we estimate
\begin{equation}\label{e:geo3}
(\forall\nnn)\quad 2M\sum_{i=0}^n\rho^i\leq\frac{2M}{1-\rho}\leq r.
\end{equation}
We now show by induction that the following holds for every $\nnn$:
\begin{subequations}
\begin{align}
&d(x_n,y_{-1})\leq \Big(\sum_{i=0}^{n}\rho^i+
\sum_{i=0}^{n-1}\rho^i\Big)M,\label{e:geo42}\\
&\max\big\{d(y_n,x_n),d(x_n,y_{n-1})\big\}\leq \rho^{n} M.\label{e:geo41}
\end{align}
\label{e:geo4}
\end{subequations}
Clearly, in view of the definition of $M$,
\eqref{e:geo4} holds when $n=0$.

Now assume that \eqref{e:geo4} holds for some $\nnn$.
First, using \eqref{e:geo3} and \eqref{e:geo4}, we have
\begin{equation}
\label{e:130427c}
\text{$d(x_n,y_{-1})\leq r$ \;\;and\;\; $d(y_n,y_{-1})\leq
d(y_n,x_n)+d(x_n,y_{-1})\leq r$.}
\end{equation}
So the contraction property applied to the quadruple $(y_{n-1},x_n,y_n,x_{n+1})$ implies
\begin{equation}\label{e:geo6}
d(x_{n+1},y_n)\leq\rho\max\{d(y_n,x_n),d(x_n,y_{n-1})\}\leq\rho^{n+1}M.
\end{equation}
It follows that
\begin{subequations}
\begin{align}
d(x_{n+1},y_{-1})&\leq d(x_{n+1},y_n)+d(y_n,x_n)+d(x_n,y_{-1})\\
&\leq \rho^{n+1}M+\rho^n M +
\Big(\sum_{i=0}^{n}\rho^i+\sum_{i=0}^{n-1}\rho^i\Big)M\\
&=\Big(\sum_{i=0}^{n+1}\rho^i+\sum_{i=0}^{n}\rho^i\Big)M\\
&\leq r.
\end{align}
\end{subequations}
So \eqref{e:geo42} holds with $n$ replaced by $n+1$.
Next, the contraction property applied to the quadruple $(x_n,y_n,x_{n+1},y_{n+1})$
yields
\begin{equation}\label{e:geo5}
d(y_{n+1},x_{n+1})\leq \rho\max\big\{d(x_{n+1},y_n),d(y_n,x_n)\big\}.
\end{equation}
In view of \eqref{e:geo5}, \eqref{e:geo6}, and \eqref{e:geo41}, we deduce
that
\begin{equation}\label{e:geo7}
\max\big\{d(y_{n+1},x_{n+1}),d(x_{n+1},y_{n})\big\}\leq \rho^{n+1}M,
\end{equation}
i.e., \eqref{e:geo41} holds with $n$ replaced by $n+1$.
Thus, by induction, \eqref{e:geo4} holds for every $\nnn$.

Combining \eqref{e:geo41}, Lemma~\ref{l:geo}, and \eqref{e:geo3},
we obtain
\begin{equation}
(\forall\nnn)\quad
  \max\{d(x_n,\bar c),d(y_n,\bar
  c)\}\leq\frac{M(1+\rho)}{1-\rho}\rho^n
  \leq\frac{r(1+\rho)}{2}\rho^n.
\end{equation}
Finally,
\eqref{e:130427c} implies
that the sequences $(x_n)_\nnn$ and $(y_n)_\nnn$,
and consequently their common limit $\bar{c}$, lie in
$\ball(y_{-1};r)$.
\end{proof}

\section{Linear Convergence of the MARP and the CQ
Condition}

\label{s:MARP}
\label{sec:CQ}

The sequences $(x_n)_\nnn$ and $(y_n)_\nnn$ produced by the MARP
need not lie in the sets $A$ and $B$, respectively. Therefore, the
techniques utilized for the method of alternating projections in
\cite{zero,one} and \cite{LLM} cannot be directly applied.  In this section,
we present a new technique which relies on the geometry of
Euclidean spaces. 

In addition to our assumptions on the sets $A$ and $B$,
the relaxation parameter sequences $\blambda = (\lambda_n)_\nnn$
and $\bmu = (\mu_n)_\nnn$,
and the MARP sequences
$\bx = (x_n)_\nnn$ and $\by=(y_n)_\nnn$ with
starting point $y_{-1}\in X$
(see \eqref{e:AandB}, \eqref{e:blbm}, and \eqref{e:bxby}),
we assume the following in this section:
\begin{subequations}
\label{e:stand}
\boxedeqn{
\text{$S$ is a subset of $X$ that is projection absorbing with respect to
$A$ and $B$, $y_{-1}\in S$,}
}
and
\boxedeqn{
(\forall\nnn)\quad
\lambda_n\geq\lambda_{n+1}\to \lambda_\infty
\;\;\text{and}\;\;
 \mu_n\geq\mu_{n+1}\to \mu_\infty,
 \quad\text{and}\quad
\alpha_\infty :=\min\{\lambda_\infty,\mu_\infty\}.
}
\end{subequations}


We start with a technical result.

\begin{lemma}\label{l:0611b}
Let $\mu\in\left]0,1\right]$ and let $\theta\in\left[0,1\right[$.
Then
\begin{equation}
  0<\mu^2+(1-\mu)^2+2\theta\mu(1-\mu)=1-2(1-\theta)\mu(1-\mu)\leq 1,
\end{equation}
and the last inequality is an equality if and only if $\mu=1$.
\end{lemma}
\begin{proof}
Clearly,
\begin{subequations}
\begin{align}
    0\leq 2(1+\theta)\mu(1-\mu)&= 2\mu(1-\mu)+2\theta\mu(1-\mu)\label{e:0611a}\\
    &\leq \mu^2+(1-\mu)^2+2\theta\mu(1-\mu)\label{e:0611b}\\
    &= \mu^2+(1-\mu)^2+2\mu(1-\mu)-2(1-\theta)\mu(1-\mu)\\
    &=(\mu+(1-\mu))^2-2(1-\theta)\mu(1-\mu)\\
    &=1-2(1-\theta)\mu(1-\mu)\leq 1\label{e:0611c}.
  \end{align}
\end{subequations}
Note that equality in \eqref{e:0611a} occurs exactly when $\mu=1$; in this
case, the inequality \eqref{e:0611b} is strict.
Furthermore, equality in \eqref{e:0611c} occurs exactly when $\mu=1$.
\end{proof}

The following result will help us later in this section to identify the convergence rate of the MARP.

\begin{lemma}\label{l:0611}
Let $\theta\in\left[0,1\right[$ and define ${\rate}\in\RP$ implicitly by
\begin{equation}\label{e:rate}
{\rate}^2:=
\sup_{\nnn}\left\{
\begin{aligned}
&\tfrac{\lambda_{n+1}^2}{\lambda_n^2}
  \big(\lambda_{n}^2+(1-\lambda_n)^2+2\theta\lambda_n(1-\lambda_n)\big),\\
&\tfrac{\mu_{n+1}^2}{\mu_n^2}
  \big(\mu_{n}^2+(1-\mu_n)^2+2\theta\mu_n(1-\mu_n)\big)
\end{aligned}\right\}.
\end{equation}
Then
$0<{\rate}\leq
\sqrt{1-2(1-\theta)\min\big\{\alpha_0(1-\alpha_0),\alpha_\infty(1-\alpha_\infty)\big\}}\leq 1$; consequently,
if $1>\alpha_0\geq\alpha_\infty>0$, then ${\rate} < 1$.
\end{lemma}
\begin{proof}
Let us first consider
\begin{equation}
  \sigma := \sup_{\nnn}\tfrac{\mu_{n+1}^2}{\mu_n^2}
  \big(\mu_{n}^2+(1-\mu_n)^2+2\theta\mu_n(1-\mu_n)\big),
\end{equation}
the corresponding supremum involving $\blambda$ is treated similarly.
Lemma~\ref{l:0611b} yields $\sigma > 0$.
Since $(\forall\nnn)$
$0<\frac{\mu_\infty}{\mu_0}\leq\frac{\mu_{n+1}}{\mu_n}\leq 1$
and hence $\frac{\mu_{n+1}^2}{\mu_n^2}\leq 1$, we estimate
with the help of Lemma~\ref{l:0611b} that
\begin{subequations}
\begin{align}
  0&<\sigma
  \leq \sup_{\nnn}\big(\mu_{n}^2+(1-\mu_n)^2+2\theta\mu_n(1-\mu_n)\big)\\
  &= \sup_{\nnn}\big(1-2(1-\theta)\mu_n(1-\mu_n)\big)\\
  &= 1 - 2(1-\theta)\min\big\{\mu_0(1-\mu_0),\mu_\infty(1-\mu_\infty)\big\}
\end{align}
\end{subequations}
because any minimizer of the function $\mu\mapsto\mu(1-\mu)$ restricted to
the interval $[\mu_\infty,\mu_0]$ must be one of the endpoints of the
interval.
The conclusion now follows by combining this estimate with its $\blambda$
counterpart.
\end{proof}

The following result provides information about the
location of limits of the MARP.

\begin{proposition}
\label{p:0602a}
Suppose both sequences $(x_n)_\nnn$ and $(y_n)_\nnn$
generated by the MARP
converge to $\bar{c}\in X$. Then the following hold:
\begin{enumerate}
\item
\label{p:0602a1}
If $\lambda_\infty>0$, then $\bar{c}\in A$.
\item
\label{p:0602a2}
If $\mu_\infty>0$, then $\bar{c}\in B$.
\item
\label{p:0602a3}
If $\alpha_\infty>0$, then $\bar{c}\in A\cap B$.
\end{enumerate}
\end{proposition}
\begin{proof}
Clearly, $x_n-y_n\to 0$ and $y_n-x_{n+1}\to 0$.

\ref{p:0602a1}:
Suppose that $\lambda_\infty>0$.
By Proposition~\ref{p:01}\ref{p:01ii},
$0\leftarrow \|x_{n+1}-y_n\| = \lambda_{n+1}d_A(y_n)$.
Since $\lambda_\infty>0$, it follows that $d_A(y_n)\to 0$.
Hence, $\bar{c}\in A$.

\ref{p:0602a2}: The proof is analogous to that of \ref{p:0602a1}.

\ref{p:0602a3}: Combine \ref{p:0602a1} and \ref{p:0602a2}.
\end{proof}

The following examples illustrate that no conclusion can be drawn about the
location of the limit point when $\lambda_\infty=0$ or  $\mu_\infty = 0$.

\begin{example}[MARP limit point lies outside $A\cup B$ and
$\lambda_\infty=\mu_\infty=0$]
\label{ex:130515c}
Suppose that $X=S=\RR$, that $A=B=\RM$.
Let $\dd\in\RPP$, and assume that
$(\blambda,\bmu)$ satisfy
\begin{equation}
\label{e:0505-1}
(\forall\nnn)\quad
\lambda_n= \mu_n = 1- \sqrt{\frac{\dd+2^{-(n+1)}}{\dd+2^{-n}}}\in\zeroun.
\end{equation}
Then $\lambda_\infty=\mu_\infty=0$.
Suppose that $y_{-1}\in\RPP=X\smallsetminus(A\cup B)$.
Then the $(\blambda,\bmu)$-MARP sequences are
\begin{equation}
(\forall\nnn)\quad x_n=y_{n-1}\sqrt{\frac{\dd+2^{-(n+1)}}{\dd+2^{-n}}}
\quad\text{and}\quad
y_n=x_n\sqrt{\frac{\dd+2^{-(n+1)}}{\dd+2^{-n}}},
\end{equation}
which inductively leads to
\begin{equation}
x_n=y_{-1}\sqrt{\frac{\dd+2^{-(n+1)}}{\dd+2^{-n}}}\left(\frac{\dd+2^{-n}}{\dd+1}\right)
\quad\text{and}\quad
y_n=y_{-1}\frac{\dd+2^{-(n+1)}}{\dd+1}.
\end{equation}
Note that $\lim_\nnn x_n = \lim_\nnn y_n=\frac{\dd y_{-1}}{\dd+1}\notin A\cup B$.
\end{example}

\begin{example}[MARP limit point lies in $A\cap B$ and
$\lambda_\infty=\mu_\infty=0$]
\label{ex:130515a}
Suppose that $X=S=\RR$, that $A=B=\RM$, and
that $\lambda_\infty=\mu_\infty=0$ while
$\sum_\nnn\lambda_n = \sum_\nnn\mu_n=\pinf$.
Furthermore, assume that $y_{-1}=\eta \in\RPP$.
Then
\begin{equation}
(\forall\nnn)\quad
x_n=\eta\prod_{i=0}^n(1-\lambda_i)\prod_{i=0}^{n-1}(1-\mu_i)
\quad\text{and}\quad
y_n=\eta\prod_{i=0}^n(1-\lambda_i)\prod_{i=0}^{n}(1-\mu_i),
\end{equation}
and so
\begin{subequations}
\begin{align}
\ln(y_n/\eta)&=\ln\Big(\prod_{i=0}^n(1-\lambda_i)\prod_{i=0}^{n}(1-\mu_i)\Big)=\sum_{i=0}^n\ln(1-\lambda_i)+ \sum_{i=0}^n\ln(1-\mu_i)\\
&\leq \sum_{i=0}^n(-\lambda_i)+\sum_{i=0}^n(-\mu_i)
\to -\infty.
\end{align}
\end{subequations}
It follows that $y_n\to 0$ and thus $x_n\to 0$.
Hence $\lim_\nnn x_n=\lim_\nnn y_n \in A\cap B$.
\end{example}

\begin{example}[MARP limit point lies outside $A\cap B$ and $\lambda_\infty>0=\mu_\infty$]
\label{ex:130515b}
Suppose that $X=S=\RR^2$, that $A=\RR\times\{0\}$, that
$B=\{0\}\times\RR$, that
$(\forall\nnn)$ $\lambda_n=\frac{1}{2}$ and
$\mu_n=1-\frac{1+2^{-(n+1)}}{1+2^{-n}}$,
and that $y_{-1}=(\eta,\zeta)\in X\smallsetminus A$.
Then
\begin{equation}
(\forall\nnn)\quad
x_n=\left(\tfrac{\eta}{2^{n+1}},\tfrac{\zeta(1+2^{-n})}{2}\right)
\quad\text{and}\quad
y_n=\left(\tfrac{\eta}{2^{n+1}},\tfrac{\zeta(1+2^{-(n+1)})}{2}\right);
\end{equation}
therefore, $\lim_\nnn x_n = \lim_\nnn y_n= (0,\tfrac{\zeta}{2}) \in B\smallsetminus
A$.
\end{example}


We now present the main convergence result of this section.

\begin{theorem}[local linear convergence]
\label{t:main1}
Let $r\in\RPP$ and let $\theta\in\left[0,1\right[$.
Assume that the following hold:
\begin{enumerate}
\item\label{t:main1-o}
$\alpha_0<1$;
\item\label{t:main1-oo}
${\rate}\leq\rho<1$, where ${\rate}$ is as in Lemma~\ref{l:0611};
\item\label{t:main1-i}
$\max\big\{d_A(y_{-1}),d_B(y_{-1})\big\}\leq
\tfrac{r(1-\rho)}{2\alpha_0(1+\alpha_0)}$;
\item\label{t:main1-ii}
\begin{equation}\label{e:main1}
\big(\forall x\in S \cap\ball(y_{-1};r)\big)\big(\forall a\in P_Ax\big)\big(\forall b\in
P_Bx\big)\quad
\scal{a-x}{x-b}\leq\theta\|a-x\|\cdot\|x-b\|.
\end{equation}
\end{enumerate}
Then the MARP sequences $(x_n)_\nnn$ and $(y_n)_\nnn$ converge linearly
with rate $\rho$
to some point $\bar c\in \ball(y_{-1};r)$ and
\begin{equation}
(\forall\nnn)\quad\max\big\{\|x_n-\bar c\|,\|y_n-\bar c\|\big\}
\leq\tfrac{r(1+\rho)}{2}\rho^n.
\end{equation}
Furthermore, if $\min\{\lambda_\infty,\mu_\infty\}>0$, then $\bar{c}\in
A\cap B$.
\end{theorem}
\begin{proof}
Combining \eqref{e:stand} and Proposition~\ref{p:1216a},
we deduce that $(x_n)_\nnn $ and $(y_n)_\nnn$ lie in $S$.

We now claim that $(x_n)_\nnn$ and $(y_n)_{n\geq -1}$ have
the alternating contraction property at $y_{-1}$ with parameters
$(r,\rho)$
(see Definition~\ref{d:contr}).
Let us fix $\nnn$ and check \eqref{e:contr} for the quadruple
$(y_{n-1},x_n,y_n,x_{n+1})$;
the other quadruple $(x_n,y_n,x_{n+1},y_{n+1})$ is treated similarly.
We assume that $\|x_n-y_{-1}\|\leq r$.
Since $a_n\in P_Ax_n$ and $b_n\in P_Bx_n$ (see Definition~\ref{d:MARP}),
\eqref{e:main1} yields
\begin{equation}
\scal{a_n-x_n}{x_n-b_n}\leq\theta\|a_n-x_n\|\cdot\|x_n-b_n\|.
\end{equation}
Proposition~\ref{p:01}\ref{p:01iii} implies that
$y_{n-1}-x_n=\frac{\lambda_n}{1-\lambda_n}(x_n-a_n)$ and
$y_n-x_n=\frac{\mu_n}{1-\mu_n}(b_n-x_n)$; thus,
\begin{equation}
\scal{y_n-x_n}{y_{n-1}-x_n}\leq\theta\|y_n-x_n\|\cdot\|x_n-y_{n-1}\|.
\end{equation}
Hence, by Lemma~\ref{l:MARP} and assumption~\ref{t:main1-oo},
\begin{equation}
\|x_{n+1}-y_{n}\|\leq \rho\max\big\{\|y_n-x_n\|,\|x_n-y_{n-1}\|\big\}.
\end{equation}
Thus \eqref{e:contr} holds, as claimed.

It now follows from \eqref{e:130425d} of Lemma~\ref{l:first} and
assumption~\ref{t:main1-i} that
\begin{equation}
M:=\max\big\{\|y_0-x_0\|,\|x_0-y_{-1}\|\big\}
\leq\alpha_0(1+\alpha_0)\max\big\{d_A(y_{-1}),d_B(y_{-1})\big\}
\leq\tfrac{r(1-\rho)}{2}.
\end{equation}
Hence, by Theorem~\ref{t:geo},
$(x_n)_\nnn$ and $(y_n)_\nnn$ converge linearly to $\bar c\in\ball(y_{-1},r)$ and
\begin{equation}
(\forall\nnn)\quad\max\big\{\|x_n-\bar c\|,\|y_n-\bar c\|\big\}
\leq\tfrac{r(1+\rho)}{2}\rho^n.
\end{equation}
Finally, recall Proposition~\ref{p:0602a}.
\end{proof}

\begin{remark}[best bound for the convergence rate]
\label{r:bb1}
In Theorem~\ref{t:main1}, the linear rate is tied to the constant
${\rate}$ defined by \eqref{e:rate}.
The computation of ${\rate}$ appears to be hard in general; however,
the upper bound provided in Lemma~\ref{l:0611} is minimized when
$\lambda_0=\lambda_\infty=\mu_0=\mu_\infty=\tfrac{1}{2}$, i.e., when
$(\forall\nnn)$ $\lambda_n=\mu_n=\tfrac{1}{2}$, in which case
\begin{equation}
0< {\rate} = \sqrt{\frac{1+\theta}{2}} < 1.
\end{equation}
\end{remark}

The following result concerns global convergence.
As a consequence, it somewhat surprisingly
guarantees the {\em nonemptiness} of the intersection.

\begin{corollary}[global convergence]\label{c:global}
Assume that
$1>\alpha_0\geq\alpha_\infty>0$ and
that there exists $\theta\in\left[0,1\right[$ such that
\begin{equation}
\label{e:130517a}
\big(\forall x\in S\big)\big(\forall a\in P_Ax\big)\big(\forall b\in P_Bx\big)
\quad \scal{a-x}{x-b}\leq\theta\|a-x\|\cdot\|x-b\|.
\end{equation}
Then the MARP sequences $(x_n)_\nnn$ and $(y_n)_\nnn$ converge linearly
with rate ${\rate}$ to some point in $A\cap B$, where
${\rate}\in\zeroun$ is defined in Lemma~\ref{l:0611}.
\end{corollary}

\begin{example}[two subspaces]
Suppose that $A$ and $B$ are affine subspaces with $A\cap B\neq\varnothing$,
that $S=\aff(A\cup B)$,
and that $1>\alpha_0\geq \alpha_\infty>0$.
Then there exists $\theta\in\left[0,1\right[$ such that ${\rate}\in\zeroun$,
where $\rate$ is defined in Lemma~\ref{l:0611}.
Moreover, the MARP sequences $(x_n)_\nnn$ and $(y_n)_\nnn$ converge linearly
with rate ${\rate}$ to some point in $A\cap B$.
\end{example}
\begin{proof}
(See also \cite[Theorem~5.7]{BBSIREV} for a closely related result.)
After translating if necessary, we assume that $A$ and $B$ are linear
subspaces, and that $S=A+B$.
Let $x\in S$, let $a\in P_Ax$, and let $b\in P_Bx$.
Using \cite[Theorem~3.5]{zero},
we have
$x-a \in \pn{A}{S}(a)\subseteq \nc{A}{S}(a) = N_A(a)\cap S = A^\perp \cap
(A+B)$.
Similarly, $x-b\in B^\perp \cap (A+B)$.
Since $A$ and $B$ are subspaces and
$A^\perp\cap (A+B)\cap B^\perp\cap (A+B)
=(A+B)\cap(A+B)^\perp = \{0\}$,
we set
\begin{equation}
\theta := \max \scal{A^\perp\cap (A^\perp\cap B^\perp)^\perp\cap \ball(0;1)}{B^\perp \cap (A^\perp \cap
B^\perp)^\perp\cap \ball(0;1)} <  1.
\end{equation}
(Thus, $\theta$ is the cosine of the \emph{Friedrichs angle} between $A^{\perp}$
and $B^{\perp}$, which is identical to the cosine of the {Friedrichs angle} between $A$ and $B$.)
Hence \eqref{e:130517a} holds and the conclusion now follows from
Corollary~\ref{c:global}.
\end{proof}

In the spirit of \cite{LLM} and \cite{one},
we now guarantee local linear convergence when the CQ condition holds.


\begin{theorem}[local convergence via CQ condition]
\label{t:loc-viaCQ}
Suppose that $1>\alpha_0\geq \alpha_\infty>0$,
that $c\in A\cap B$ and
that the $(A,S,B,S)$-CQ holds at $c$, i.e. (see
Definition~\ref{d:CQ}),
\begin{equation}
\nc{A}{S}(c)\cap(-\nc{B}{S}(c))=\{0\}.
\end{equation}
In view of \eqref{e:CQ-CQn},
the limiting CQ number associated with $(A,S,B,S)$
(see Definition~\ref{d:CQn}) satisfies
\begin{equation}
\overline\theta=\max\Menge{\scal{u}{v}}{ u\in \nc{A}{S}(c), v\in
-\nc{B}{S}(c),\|u\|\leq 1, \|v\|\leq 1} < 1.
\end{equation}
Let $\theta\in\left]\overline\theta,1\right[$.
Then there exists $\dd>0$ such that
whenever the starting point $y_{-1}$ lies in $S\cap\ball(c;\dd)$,
the sequences $(x_n)_\nnn$ and $(y_n)_\nnn$ generated by the MARP
converge linearly to a point in $A\cap B$
with rate ${\rate}\in\zeroun$ (see Lemma~\ref{l:0611}).
\end{theorem}
\begin{proof}
There exists $\ve>0$ sufficiently small such that
$\theta_{2\ve}\leq\theta$,
where $\theta_{2\ve}$ is the CQ number associated with
$(A,S,B,S)$ and $2\ve$ (see Definition~\ref{d:CQn}).
We claim that
\begin{equation}\label{e:ve-to-dd}
\dd:=\tfrac{\ve(1-\rate)}{1-{\rate}+2\alpha_0(1+\alpha_0)}
\end{equation}
does the job.

To this end, assume that $y_{-1}\in S\cap\ball(c;\dd)$ and set
\begin{equation}
\label{e:130426a}
r:=\tfrac{2\dd\alpha_0(1+\alpha_0)}{1-\rate}.
\end{equation}
Since $c\in A\cap B$, we deduce that
\begin{equation}
\max\big\{d_A(y_{-1}), d_B(y_{-1})\big\}\leq \|y_{-1}-c\|\leq\dd
=\tfrac{r(1-\rate)}{2\alpha_0(1+\alpha_0)},
\end{equation}
which is assumption \ref{t:main1-i} of  Theorem~\ref{t:main1}.

Now let $x\in S\cap \ball(y_{-1};r)$,
let $a\in P_Ax$, and let $b\in P_Bx$.
Using \eqref{e:ve-to-dd} and \eqref{e:130426a}, we estimate
\begin{equation}
\|x-c\|\leq\|x-y_{-1}\|+\|y_{-1}-c\|\leq r + \dd = \ve.
\end{equation}
Hence, $\|a-c\|\leq \|a-x\|+\|x-c\|=d_A(x)+\|x-c\|\leq
2\|x-c\|\leq 2\ve$.
Analogously, $\|b-c\|\leq 2\ve$.
On the other hand, $a-x\in-\pn{A}{S}(a)$ and $x-b\in\pn{B}{S}(b)$.
It thus follows from the definition of the CQ-number
(see \eqref{e:CQn}) and our choice of $\ve$ that
\begin{equation}
\scal{a-x}{x-b}\leq\theta_{2\ve}\|a-x\|\cdot\|x-b\|\leq\theta\|a-x\|\cdot\|x-b\|,
\end{equation}
which is assumption~\ref{t:main1-ii} of Theorem~\ref{t:main1}.
Therefore, Theorem~\ref{t:main1} implies that
$(x_n)_\nnn$ and $(y_n)_\nnn$ converge linearly
to a point $\bar c\in A\cap B\cap\ball(y_{-1};r)$ and
\begin{equation}
(\forall\nnn)\quad\max\big\{\|x_n-\bar c\|,\|y_n-\bar c\|\big\}
\leq\tfrac{r(1+\rate)}{2}\rate^n
=\tfrac{\ve\alpha_0(1+\alpha_0)(1+\rate)}{1-\rate+2\alpha_0(1+\alpha_0)}\rate^n.
\end{equation}
We also note that $\bar c\in\ball(c;\ve)$ because
$\|\bar c-c\|\leq\|\bar c-y_{-1}\|+\|y_{-1}-c\|\leq r+\dd=\ve$.
\end{proof}

Finally, we use
Aharoni and Censor's \cite[Theorem~1]{AhaCen} to obtain
a linear convergence rate result in the convex case.

\begin{corollary}[two convex sets]
Suppose that $A$ and $B$ are convex
with $\reli A \cap \reli B \neq\varnothing$,
that $S=\aff(A\cup B)$, and
that $1>\alpha_0\geq \alpha_\infty>0$.
Then the sequences $(x_n)_\nnn$ and $(y_n)_\nnn$ generated by the MARP
converge linearly to a point in $A\cap B$.
\end{corollary}
\begin{proof}
It is known that the MARP sequences converge to some point $c\in A \cap B$;
see, e.g., the aforementioned \cite[Theorem~1]{AhaCen}.
By \cite[Proposition~7.5]{zero}, the $(A,S,B,S)$-CQ condition holds at $c$.
In view of \eqref{e:CQ-CQn},
the limiting CQ number associated with $(A,S,B,S)$
(see Definition~\ref{d:CQn}) satisfies
\begin{equation}
\overline\theta=\max\Menge{\scal{u}{v}}{ u\in \nc{A}{S}(c), v\in
-\nc{B}{S}(c),\|u\|\leq 1, \|v\|\leq 1} < 1.
\end{equation}
Let $\theta\in\left]\overline\theta,1\right[$
and obtain $\dd>0$ as in Theorem~\ref{t:loc-viaCQ}.
Since $(x_n)_\nnn$ and $(y_n)_\nnn$ converge to $c$,
there exists $n_0\in\NN$
such that $y_{n_0}\in B(c;\dd)$.
The conclusion therefore follows from Theorem~\ref{t:loc-viaCQ}
(applied to the MARP with starting point $y_{n_0}\in S$).
\end{proof}


\section{Linear Convergence of the MARP and Regularity}
\label{s:2MARP}
\label{sec:regu}

We now investigate the MARP in the presence of regularity.
We uphold the assumptions \eqref{e:stand} of the previous
section.

The following result is a counterpart of Lemma~\ref{l:MARP};
it refines \cite[Theorem~5.2]{LLM} and \cite[Proposition~3.4]{one}.

\begin{lemma}\label{l:2MARP}
Let $\theta\in\left[0,1\right[$,
let $\dd>0$,
let $\ve\geq 0$
and let $\nnn$.
Suppose that $c\in A$,
that $A$ is $(S,\ve,2\dd)$-regular at $c$
(see Definition~\ref{d:reg}),
and that
the quadruple $(y_{n-1},x_n,y_n,x_{n+1})$ generated by the
MARP (see Definition~\ref{d:MARP}) with starting point
$y_{-1}\in S$ satisfies
\begin{equation}\label{e:2MARP01}
\{x_n,y_n\}\subseteq\ball(c;\dd)
\quad\text{and}\quad
\scal{x_{n+1}-y_n}{x_n-y_n}\leq \theta\|x_{n+1}-y_n\|\cdot\|x_n-y_n\|.
\end{equation}
Then
\begin{equation}\label{e:2MARP02}
\|x_{n+1}-y_n\|\leq
\tfrac{\lambda_{n+1}}{\lambda_n}
\big(\theta\lambda_n+2\ve+1-\lambda_n\big)
\max\big\{\|y_n-x_n\|,\|x_n-y_{n-1}\|\big\}.
\end{equation}
Furthermore, the following implication holds:
\begin{equation}
\label{e:2MARP02s}
\lambda_n=1\;\Rightarrow\;\|x_{n+1}-y_n\|\leq
\lambda_{n+1}(\theta+2\ve)\|y_n-x_n\|.
\end{equation}
\end{lemma}
\begin{proof}
Using Proposition~\ref{p:01}\ref{p:01i}, we have
$\|a_n-c\|\leq\|x_n-a_n\|+\|x_n-c\|=d_A(x_n)+\|x_n-c\|
\leq 2\|x_n-c\|\leq 2\dd$. Moreover,
$\|a_{n+1}-c\|\leq \|a_{n+1}-y_n\|+\|y_n-c\|
=d_A(y_n)+\|y_n-c\|\leq 2\|y_n-c\|\leq 2\dd$.
Since $y_n-a_{n+1}\in\pn{A}{S}(a_{n+1})$ and
$A$ is $(S,\ve,2\dd)$-regular at $c$, we obtain
\begin{equation}\label{e:0601-1}
  \scal{a_{n+1}-y_n}{a_{n+1}-a_n}\leq\ve\|a_{n+1}-y_n\|\cdot\|a_{n+1}-a_n\|.
\end{equation}
Now
$a_{n+1}-y_n=\frac{1}{\lambda_{n+1}}(x_{n+1}-y_n)$
(by Proposition~\ref{p:01}\ref{p:01ii})
and \eqref{e:2MARP01} imply
\begin{equation}\label{e:0601-2}
  \scal{a_{n+1}-y_n}{x_n-y_n}\leq\theta\|a_{n+1}-y_n\|\cdot\|x_n-y_n\|.
\end{equation}
Adding \eqref{e:0601-1}, \eqref{e:0601-2}
and $\scal{a_{n+1}-y_n}{a_n-x_n}\leq \|a_{n+1}-y_n\|\cdot\|a_n-x_n\|$,
we obtain
\begin{equation}
\begin{aligned}
  \|a_{n+1}-y_n\|^2\leq&\ \theta\|a_{n+1}-y_n\|\cdot\|x_n-y_n\|\\
  &+\ve\|a_{n+1}-y_n\|\cdot\|a_{n+1}-a_n\| + \|a_{n+1}-y_n\|\cdot\|a_n-x_n\|;
\end{aligned}
\end{equation}
thus,
\begin{equation}\label{e:0601-3}
\|a_{n+1}-y_n\|\leq \theta\|y_n-x_n\|+ \ve\|a_{n+1}-a_n\| + \|a_n-x_n\|.
\end{equation}
Substituting
$\|a_{n+1}-a_n\|\leq\|a_{n+1}-y_n\|+\|a_n-y_n\|\leq 2\|a_n-y_n\|
\leq 2\|y_n-x_n\|+2\|x_n-a_n\|$ into \eqref{e:0601-3}
results in
\begin{equation}
  \|a_{n+1}-y_n\|\leq (\theta+2\ve)\|y_n-x_n\|+ (1+2\ve)\|a_n-x_n\|.
\end{equation}
Therefore, since
$\|a_n-x_n\|=\frac{1-\lambda_n}{\lambda_n}\|x_n-y_{n-1}\|$ and $\|a_{n+1}-y_n\|=\frac{1}{\lambda_{n+1}}\|x_{n+1}-y_n\|$
by Proposition~\ref{p:01}\ref{p:01iii}\&\ref{p:01ii},
we obtain
\begin{subequations}
\begin{align}
\|x_{n+1}-y_n\|&\leq \tfrac{\lambda_{n+1}}{\lambda_n}\big((\theta+2\ve)\lambda_n\|y_n-x_n\|+ (1+2\ve)(1-\lambda_n)\|x_n-y_{n-1}\|\big)
\label{e:2MARP03a}\\
&\leq
\tfrac{\lambda_{n+1}}{\lambda_n}\big((\theta+2\ve)\lambda_n+
(1+2\ve)(1-\lambda_n)\big)\max\big\{\|y_n-x_n\|,\|x_n-y_{n-1}\|\big\}\\
&=\tfrac{\lambda_{n+1}}{\lambda_n}\big(\theta\lambda_n+2\ve+1-\lambda_n
\big) \max\big\{\|y_n-x_n\|,\|x_n-y_{n-1}\|\big\},
\end{align}
\end{subequations}
which is \eqref{e:2MARP02}, as announced.
Finally, \eqref{e:2MARP02s} follows from \eqref{e:2MARP03a}.
\end{proof}

Analogously to the proof of Lemma~\ref{l:2MARP}, we obtain
the following result.

\begin{lemma}\label{l:2MARP'}
Let $\theta\in\left[0,1\right[$,
let $\dd>0$,
let $\ve\geq 0$
and let $\nnn$.
Suppose that $c\in B$,
that $B$ is $(S,\ve,2\dd)$-regular at $c$
(see Definition~\ref{d:reg}),
and that
the quadruple $(x_{n},y_n,x_{n+1},y_{n+1})$ generated by the
MARP (see Definition~\ref{d:MARP}) with starting point
$y_{-1}\in S$ satisfies
\begin{equation}
\{y_n,x_{n+1}\}\subseteq\ball(c;\dd)
\quad\text{and}\quad
\scal{y_{n+1}-x_{n+1}}{y_n-x_{n+1}}\leq \theta\|y_{n+1}-x_{n+1}\|
\cdot\|y_n-x_{n+1}\|.
\end{equation}
Then
\begin{equation}
\|y_{n+1}-x_{n+1}\|\leq
\tfrac{\mu_{n+1}}{\mu_n}
\big(\theta\mu_n+2\ve+1-\mu_n\big)
\max\big\{\|x_{n+1}-y_n\|,\|y_n-x_{n}\|\big\}.
\end{equation}
Furthermore, the following implication holds:
\begin{equation}
\label{e:2MARP02s'}
\mu_n=1\;\Rightarrow\;\|y_{n+1}-x_{n+1}\|\leq
\mu_{n+1}(\theta+2\ve)\|x_{n+1}-y_n\|.
\end{equation}
\end{lemma}


The next result will be useful later in this section.

\begin{lemma}
\label{l:rts}
Assume that $\alpha_\infty > 0$ and
let $\ve\in\RPP$ and $\theta\in\left[0,1\right[$ be such that
$(1-\theta)\alpha_\infty > 2\ve$.
Then
\begin{subequations}
\label{e:rts}
\begin{align}
0&<\rts:=\sup_{\nnn}
\Big\{
\tfrac{\lambda_{n+1}}{\lambda_n}\big(\theta\lambda_n+2\ve+1-\lambda_n\big),
\tfrac{\mu_{n+1}}{\mu_n}\big(\theta\mu_n+2\ve+1-\mu_n\big)
\Big\}\\
&\leq 1-\big((1-\theta)\alpha_\infty-2\ve\big)<1.
\end{align}
\end{subequations}
\end{lemma}
\begin{proof}
Clearly, $0<\rts$.
Since $\theta-1<0$, we obtain
\begin{equation}
(\forall\nnn)\quad
\theta\lambda_n+2\ve+1-\lambda_n=(\theta-1)\lambda_n+1+2\ve
\leq(\theta-1)\alpha_\infty+1+2\ve
\end{equation}
and $\theta\mu_n+2\ve+1-\mu_n\leq(\theta-1)\alpha_\infty+1+2\ve$.
Therefore, $\rts\leq(\theta-1)\alpha_\infty+1+2\ve<1$.
\end{proof}


The proof of the following result is partially similar to that of
Theorem~\ref{t:main1};
however, the linear rates of convergence obtained are different.

\begin{theorem}[MARP with regularity of sets]\label{t:main2}
Let $\ve\geq 0$, $\dd>0$, and
$\theta\in\left[0,1-2\ve\right[$.
Assume that the following hold:
\begin{enumerate}
\item
\label{t:2MARP-i}
$A$ and $B$ are $(S,\ve,2\dd)$-regular at $c \in A\cap B$;
\item
\label{t:2MARP-ii}
$\displaystyle \big(\forall x\in S \cap\ball(c;\dd)\big)\big(\forall a\in P_Ax\big)\big(\forall b\in
P_Bx\big)\quad
\scal{a-x}{x-b}\leq\theta\|a-x\|\cdot\|x-b\|$;
\item
\label{t:2MARP-iii}
$\alpha_\infty > 2\ve/(1-\theta)\geq 0$.
\end{enumerate}
Assume also that the starting point $y_{-1}$ of the
MARP sequences
$(x_n)_\nnn$ and $(y_n)_\nnn$ satisfies
\begin{equation}\label{e:2m2}
y_{-1}\in S\quad\text{and}\quad
\|y_{-1}-c\|\leq
\frac{\dd(1-\rts)}{1-\rts+2\alpha_0(1+\alpha_0)}\ ,
\end{equation}
where $\rts\in\zeroun$ is as in \eqref{e:rts}.
Then $(x_n)_\nnn$ and $(y_n)_\nnn$ converge linearly
to a point $\bar c\in A\cap B\cap\ball(c;\dd)$ with rate $\rts$;
indeed,
\begin{equation}
(\forall\nnn)\quad
\max\big\{\|x_n-\bar c\|,\|y_n-\bar c\|\big\}
  \leq\frac{\dd\alpha_0(1+\alpha_0)(1+\rts)}
  {1-\rts+2\alpha_0(1+\alpha_0)}\rts^n.
\end{equation}
Furthermore, if $(\forall\nnn)$ $\lambda_n=\mu_n=1$,
then $(x_n)_\nnn$ and $(y_n)_\nnn$ converge linearly with
rate $\rts^2=(\theta+2\ve)^2$:
\begin{equation}
(\forall\nnn)\quad\max\big\{\|x_n-\bar c\|,\|y_n-\bar c\|\big\}\leq
\frac{2\dd(1+\rts^2)}{(1+\rts)(5-\rts)}\rts^{2n}.
\end{equation}
\end{theorem}
\begin{proof}
Set
\begin{equation}
\label{e:130427a}
r:=\frac{2\dd\alpha_0(1+\alpha_0)}{1- \rts+2\alpha_0(1+\alpha_0)}.
\end{equation}
We claim that
\begin{equation}
\label{e:130427b}
\text{$(x_n)_\nnn$ and $(y_n)_{\nnn}$ have
the alternating contraction property}
\end{equation}
at $y_{-1}$ with parameter $(r,\rts)$ (recall Definition~\ref{d:contr}).
Let $\nnn$ and consider first the quadruple $(y_{n-1},x_n,y_n,x_{n+1})$.
In order to prove \eqref{e:contr}, we start by assuming that
\begin{equation}
\label{e:130427d}
\max\big\{\|x_n-y_{-1}\|,\|y_n-y_{-1}\|\big\}\leq r.
\end{equation}
Then, using \eqref{e:2m2} and \eqref{e:130427a}, we obtain
\begin{equation}\label{e:2m4}
\max\big\{\|x_n-c\|,\|y_n-c\|\big\}\leq r+\|y_{-1}-c\|\leq r +
\frac{\dd(1-\rts)}{1-\rts+2\alpha_0(1+\alpha_0)}=\dd.
\end{equation}
Applying \ref{t:2MARP-ii} with
$y_n$, $a_{n+1}\in P_A y_n$, and $b_n= P_B y_n$, we see that
\begin{equation}
\scal{a_{n+1}-y_n}{y_n-b_n}\leq\theta\|a_{n+1}-y_n\|\cdot\|y_n-b_n\|.
\end{equation}
On the other hand, Proposition~\ref{p:01}\ref{p:01ii}\&\ref{p:01iii}
implies
$a_{n+1}-y_n=\frac{1}{\lambda_{n+1}}(x_{n+1}-y_n)$ and
$y_n-b_n=\frac{1-\mu_n}{\mu_n}(x_n-y_n)$.
Altogether,
\begin{equation}\label{e:2m5}
\scal{x_{n+1}-y_n}{x_n-y_n}\leq\theta\|x_{n+1}-y_n\|\cdot\|x_n-y_n\|.
\end{equation}
In view of Lemma~\ref{l:2MARP}, we now deduce
\begin{equation}\label{e:2m3}
\|x_{n+1}-y_{n}\|\leq
\rts\max\big\{\|y_n-x_n\|,\|x_n-y_{n-1}\|\big\}.
\end{equation}
This verifies \eqref{e:contr} for the quadruple
$(y_{n-1},x_n,y_n,x_{n+1})$.
The quadruple $(x_n,y_n,x_{n+1},y_{n+1})$
is treated similarly (invoke Lemma~\ref{l:2MARP'} instead of
Lemma~\ref{l:2MARP}).
Therefore, \eqref{e:130427b} holds.

Next, using inequality \eqref{e:130425d} of Lemma~\ref{l:first},
the assumption that $c\in A\cap B$ (see \ref{t:2MARP-i}),
and \eqref{e:2m2},
we obtain
\begin{subequations}
\label{e:130428a}
\begin{align}
\max\big\{\|y_0-x_0\|,\|x_0-y_{-1}\|\big\}
&\leq\alpha_0(1+\alpha_0)\max\big\{d_A(y_{-1}),d_B(y_{-1})\big\}\\
&\leq\alpha_0(1+\alpha_0)\|y_{-1}-c\|\\
&\leq\frac{\alpha_0(1+\alpha_0)\dd(1-\rts)}{1-\rts+2\alpha_0(1+\alpha_0)}\\
&=\frac{r(1-\rts)}{2}.
\end{align}
\end{subequations}
Thus, Theorem~\ref{t:geo} and \eqref{e:130427a} yield
the existence of $\bar{c}\in\ball(y_{-1};r)$ such that
\begin{equation}
(\forall\nnn)\quad\max\big\{\|x_n-\bar c\|,\|y_n-\bar c\|\big\}
\leq\frac{r(1+\rts)}{2}\rts^n
=\frac{\dd\alpha_0(1+\alpha_0)(1+\rts)}
  {1-\rts+2\alpha_0(1+\alpha_0)}\rts^n.
\end{equation}
Furthermore, Theorem~\ref{t:geo} also states
that \eqref{e:130427d} holds for every $\nnn$;
consequently, so does its consequence \eqref{e:2m4}.
Also, Proposition~\ref{p:0602a}, assumption~\ref{t:2MARP-iii}
and \eqref{e:2m4} imply that $\bar{c}\in A\cap B\cap\ball(c;\dd)$.

Finally, we additionally assume that
$(\forall\nnn)$ $\lambda_n=\mu_n=1$.
Then $\alpha_0=\alpha_\infty=1$,
$\rts=\theta+2\ve$, and $r=\frac{4\dd}{5-\rts}$.
Combining \eqref{e:2m4}, \eqref{e:2m5},
\eqref{e:2MARP02s} and
\eqref{e:2MARP02s'}
yields
\begin{equation}
(\forall\nnn)\quad
\|x_{n+1}-y_{n}\|\leq \rts\|y_n-x_n\|
\;\text{and}\;
 \|y_n-x_n\|\leq \rts\|x_n-y_{n-1}\|;
\end{equation}
consequently, $\|x_{n+1}-y_{n}\|\leq  \rts^2\|x_n-y_{n-1}\|
=\rts^2\max\{\|y_n-x_n\|,\|x_n-y_{n-1}\|\}$ and similarly
$\|y_{n+1}-x_{n+1}\|^2\leq\rts^2\max\{\|x_{n+1}-y_n\|,\|y_n-x_n\|\}$.
Thus, $(x_n)_\nnn$ and $(y_n)_\nnn$ have the contraction property at
$y_{-1}$ with parameters $(r,\rts^2)$.
Now, \eqref{e:geo1} holds with $(r,\rts^2)$ because
of \eqref{e:130428a} and
\begin{equation}\label{e:2m6}
M:=\max\big\{\|y_0-x_0\|,\|x_0-y_{-1}\|\big\}\leq
\tfrac{r(1-\rts)}{2}\leq\tfrac{r(1-\rts^2)}{2}.
\end{equation}
Hence, Theorem~\ref{t:geo} implies that $(x_n)_\nnn$ and $(y_n)_\nnn$ converge
linearly with rate $\rts^2=(\theta+2\ve)^2$; in fact,
\begin{equation}
(\forall\nnn)\quad\max\big\{\|x_n-\bar c\|,\|y_n-\bar c\|\big\}\leq
\tfrac{M(1+\rts^2)}{1-\rts^2}\rts^{2n}\leq
\tfrac{2\dd(1+\rts^2)}{(1+\rts)(5-\rts)}\rts^{2n}.
\end{equation}
This completes the proof.
\end{proof}

\begin{remark}[comparing rates in the $(S,0,2\dd)$-regular case]
Consider Theorem~\ref{t:main2} when $A$ and $B$
are $(S,0,2\dd)$-regular at $c\in A\cap B$,
and $\alpha_\infty>0$.
This happens, e.g., when $A$ and $B$ are convex.
Since $\ve=0$, we have $\theta\in\left[0,1\right[$.
Consider the function
\begin{equation}
f\colon \left]0,1\right]\to\zeroun\colon \lambda\mapsto
\frac{\theta\lambda+1-\lambda}{\sqrt{\lambda^2+(1-\lambda)^2+2\theta\lambda(1-\lambda)}}.
\end{equation}
Then
\begin{equation}
f'\colon \lambda \mapsto
\frac{-\lambda(1-\theta^2)}{\big(\lambda^2+(1-\lambda)^2+2\theta\lambda(1-\lambda)\big)^{3/2}}<0.
\end{equation}
Hence $f$ is strictly decreasing and therefore
$(\forall\nnn)$ $f(\alpha_\infty)\geq f(\lambda_\infty)\geq f(\lambda_n)$;
consequently,
\begin{subequations}
\label{e:130515c}
\begin{equation}
(\forall\nnn)\quad
\tfrac{\lambda_{n+1}}{\lambda_n}\big(\theta\lambda_n+1-\lambda_n\big)
\leq f(\alpha_\infty) \tfrac{\lambda_{n+1}}{\lambda_n}
  \sqrt{\lambda_{n}^2+(1-\lambda_n)^2+2\theta\lambda_n(1-\lambda_n)}.
\end{equation}
Similarly,
\begin{equation}
(\forall\nnn)\quad
\tfrac{\mu_{n+1}}{\mu_n}\big(\theta\mu_n+1-\mu_n\big)
\leq f(\alpha_\infty) \tfrac{\mu_{n+1}}{\mu_n}
  \sqrt{\mu_{n}^2+(1-\mu_n)^2+2\theta\mu_n(1-\mu_n)}.
\end{equation}
\end{subequations}
On the other hand,
\begin{subequations}
$\rts$ defined in \eqref{e:rts} becomes
\begin{equation}
\rts=\sup_{\nnn}
\Big\{
\tfrac{\lambda_{n+1}}{\lambda_n}\big(\theta\lambda_n+1-\lambda_n\big),
\tfrac{\mu_{n+1}}{\mu_n}\big(\theta\mu_n+1-\mu_n\big)
\Big\}
\end{equation}
while $\rate$ defined by \eqref{e:rate} satisfies
\begin{equation}
{\rate}=
\sup_{\nnn}\left\{
\begin{aligned}
&\tfrac{\lambda_{n+1}}{\lambda_n}
  \sqrt{\lambda_{n}^2+(1-\lambda_n)^2+2\theta\lambda_n(1-\lambda_n)},\\
&\tfrac{\mu_{n+1}}{\mu_n}
  \sqrt{\mu_{n}^2+(1-\mu_n)^2+2\theta\mu_n(1-\mu_n)}
\end{aligned}\right\}.
\end{equation}
\end{subequations}
Altogether,
\begin{equation}
\rts \leq f(\alpha_\infty) \rate < \rate.
\end{equation}
Therefore, the rate $\rts$ is always better than the rate
$\rate$.
\end{remark}

\begin{remark}[best bound for the convergence rate]
\label{r:bb2}
In Theorem~\ref{t:main2}, the linear rate is bounded above
by the constant $\rts^2$ defined in \eqref{e:rts}.
Again, the actual computation of $\rts$ seems to be hard in
general; however, the upper bound in Lemma~\ref{l:rts}
is minimized when $(\forall\nnn)$ $\lambda_n=\mu_n=1$,
in which case
\begin{equation}
\rts^2 = (\theta+2\ve)^2.
\end{equation}
Comparing to the best bound derived in Remark~\ref{r:bb1},
we note that
for fixed $\theta\in\left[0,1\right[$ and
for all $\ve>0$ sufficiently small
\begin{equation}
(\theta+2\ve)^2 < \theta+2\ve < \sqrt{\theta} <
\sqrt{\frac{1+\theta}{2}}.
\end{equation}
Thus, when $\theta\to 1^-$, we expect the linear rate of
convergence for the MARP to approach that of the unrelaxed MAP.
\end{remark}


\section{MARP with linearly vanishing relaxation parameters}

\label{sec:vanish}

In this section, we consider the
$(\blambda,\bmu)$-MARP sequences with linearly vanishing
relaxation parameters; specifically, we assume that
\boxedeqn{
\label{e:3stand}
\Rts := \sup_{\nnn}\left\{
\frac{\lambda_{n+1}}{\lambda_n}, \frac{\mu_{n+1}}{\mu_n}\right\}<1.
}
A concrete instance occurs when
$(\forall\nnn)$ $\lambda_n=\lambda_0\Rts^n$ and
$\mu_n=\mu_0\Rts^n$.

The following result guarantees that the MARP sequences
{\em always converge linearly and globally without any assumption
on regularity or CQ-type conditions whatsoever.}

\begin{theorem}\label{t:0224a}
The MARP sequences $(x_n)_\nnn$ and $(y_n)_\nnn$
converge linearly to some point $\bar{c}\in X$ with rate
$\Rts$; moreover,
\begin{equation}
\label{e:130428d}
(\forall\nnn)\quad\max\big\{\|x_n-\bar c\|,\|y_n-\bar c)\|\big\}
\leq \tfrac{M(1+\Rts)}{1-\Rts}\cdot \Rts^n,
\end{equation}
where $M:=\max\{\|y_0-x_0\|,\|x_0-y_{-1}\|\}$, and
\begin{equation}
\label{e:130428e}
\|\bar c -y_{-1}\|\leq \tfrac{2\alpha_0(1+\alpha_0)}{1-\Rts}
\max\big\{d_A(y_{-1}),d_B(y_{-1})\big\}.
\end{equation}
\end{theorem}
\begin{proof}
Let $\nnn$.
Clearly,
$\scal{y_n-x_n}{y_{n-1}-x_n}\leq 1\cdot
\|y_n-x_n\|\cdot\|x_n-y_{n-1}\|$ because of Cauchy-Schwarz.
Lemma~\ref{l:MARP} (applied with $\theta=1$) yields
\begin{subequations}
\label{e:130428b}
\begin{align}
\|x_{n+1}-y_n\| &\leq\tfrac{\lambda_{n+1}}{\lambda_n}
\max\big\{\|y_n-x_n\|,\|x_n-y_{n-1}\|\big\}\\
&\leq \Rts\max\big\{\|y_n-x_n\|,\|x_n-y_{n-1}\|\big\}.
\end{align}
\end{subequations}
On the other hand, by using Lemma~\ref{l:MARP'}, we similarly obtain
\begin{equation}
\label{e:130428c}
\|y_{n+1}-x_{n+1}\|\leq
\Rts\max\big\{\|x_{n+1}-y_n\|,\|y_n-x_n\|\big\}.
\end{equation}
Altogether,
\begin{subequations}
\begin{align}
\max\big\{\|y_{n+1}-x_{n+1}\|,\|x_{n+1}-y_{n}\|\big\}
&\leq\Rts\max\big\{\|y_n-x_n\|,\|x_n-y_{n-1}\|\big\}\\
&\;\;\vdots \\
&\leq\Rts^{n+1}\max\big\{\|y_0-x_0\|,\|x_0-y_{-1}\|\big\}\\
&=M \Rts^{n+1}.
\end{align}
\end{subequations}
Thus
\begin{equation}\label{e:0224a1}
(\forall\nnn)\quad
\max\big\{\|y_{n}-x_{n}\|,\|x_{n}-y_{n-1}\|\big\}\leq M \Rts^{n}
\end{equation}
because \eqref{e:0224a1} holds for $n=0$ by the definition of $M$.
Therefore, by Lemma~\ref{l:geo}, there exists a point $\bar{c}\in
X$ such that
\begin{equation}
(\forall\nnn)\quad\max\big\{\|x_n-\bar c\|,\|y_n-\bar c)\|\big\}\leq
\tfrac{M(1+\Rts)}{1-\Rts}\cdot \Rts^n,
\end{equation}
i.e., \eqref{e:130428d} holds.
In particular,
$\|x_0-\bar c\|\leq \frac{M(1+\Rts)}{1-\Rts}$
and thus
\begin{equation}
\|\bar c - y_{-1}\|\leq \|x_0-y_{-1}\|+\|x_0-\bar c\|
\leq M+\tfrac{M(1+\Rts)}{1-\Rts}=\tfrac{2M}{1-\Rts}.
\end{equation}
On the other hand, Lemma~\ref{l:first} yields
$M\leq \alpha_0(1+\alpha_0)\max\{d_A(y_{-1}),d_B(y_{-1})\}$.
Altogether, we obtain \eqref{e:130428e}.
\end{proof}

\begin{remark}
It is interesting to compare the results of this section to
some of the results of previous sections.
On the one hand,
Theorem~\ref{t:0224a} yields universal and global linear
convergence; however, the location of the limit is not known to
be in the intersection $A\cap B$.
On the other hand,
Theorem~\ref{t:loc-viaCQ} and Theorem~\ref{t:main2} guarantee
linear convergence when a CQ condition or regularity holds,
respectively; nevertheless, these results are only local.
We appear to witness here an ``uncertainty principle'' which pits
quality of convergence against location of the limit.
It would be highly desirable to design hybrid methods that
guarantee global convergence to a point in the intersection
(or to prove that such an undertaking is hopeless).
\end{remark}


\section{Further Examples}

\label{sec:examples}

\begin{proposition}
\label{p:130514}
Suppose that $X=\RR$ and let $(a,b,c)\in\RPP\times\RPP\times\RMM$ satisfy
\begin{equation}
\label{e:130514a}
\max\{a,b-2a\} < |c| = -c < \sqrt{a^2 + (b-a)^2} < b.
\end{equation}
Suppose that $A=\{a,c\}$, that $B=\{b,c\}$,
that
\begin{equation}
\label{e:130514b}
(\forall\nnn)\quad \lambda_n=\mu_n = \lambda \in
\left] \frac{a+c+\sqrt{c^2-a^2}}{2a},\frac{b+c}{2a}\right[ \subset \zeroun
\end{equation}
and that $y_{-1}=0$.
Then the following hold:
\begin{enumerate}
\item
\label{p:130514i}
The MAP cycles between $a$ and $b$, and thus does not converge to a point in $A \cap B$.
\item
\label{p:130514ii}
The $\lambda$-MARP converges linearly to $c\in A\cap B$ with rate $1-\lambda$.
\end{enumerate}
\end{proposition}
\begin{proof}
On the one hand,
$c^2 < a^2 + (b-a)^2$
$\Leftrightarrow$
$(b-a)^2 > c^2-a^2$
$\Leftrightarrow$
$b-a=|b-a|>\sqrt{c^2-a^2}$
$\Leftrightarrow$
\begin{equation}
b+c>a+c+\sqrt{c^2-a^2} = \sqrt{|c|-a}\Big(\sqrt{|c|+a}-\sqrt{|c|-a}\Big)>0.
\end{equation}
On the other hand,
$2a > b+c>0$. Altogether, the interval from which $\lambda$ is drawn is
well defined and we have
\begin{equation}
\label{e:130430aa}
b+c > 2a\lambda > a+c+\sqrt{c^2-a^2}.
\end{equation}
Since $a=|a|<|c|$, it follows that $P_Ay_{-1}=P_A0=a$.
Hence
\begin{equation}
x_0 = \lambda a.
\end{equation}
Now $b-2a < |c|$
$\Leftrightarrow$
$b-a<a+|c|$
$\Leftrightarrow$
$|b-a| < a-c$
$\Leftrightarrow$
$|b-a| < |c-a|$,
so $P_Ba = b$
and obviously $P_Ab = a$. This proves \ref{p:130514i}.

Next,
$P_Bx_0 = c$
$\Leftrightarrow$
$|c-\lambda a| < |b-\lambda a|$
$\Leftrightarrow$
$\lambda a+|c| < b-\lambda a$
$\Leftrightarrow$
$2\lambda a < b - |c| = b+ c$.
By the first inequality in \eqref{e:130430aa}, $P_Bx_0=c$ and thus
$y_0 = (1-\lambda)x_0 + \lambda P_Bx_0
= (1-\lambda)\lambda a + \lambda c$, i.e.,
\begin{equation}
y_0 = (1-\lambda)\lambda a + \lambda c.
\end{equation}
We have $P_Ay_0 = c$ if and only if
$y_0 < (c+a)/2$, which is equivalent to
$2(1-\lambda)\lambda a + 2\lambda c < a+c$.
Viewed in terms of $\lambda$, this is a quadratic inequality which
holds because of the second inequality in \eqref{e:130430aa}.
It follows that $x_1 = (1-\lambda)y_0 + \lambda P_Ay_0$,
i.e.,
\begin{equation}
x_1 = (1-\lambda)\big( (1-\lambda)\lambda a + \lambda c\big)+\lambda c.
\end{equation}
Furthermore, $x_0-x_1 = \lambda a(2-\lambda)(\lambda - c/a)>0$
and so $x_1<x_0$. Since already $P_Bx_0=c$, it follows that $P_Bx_1=c$ and
therefore
\begin{equation}
y_1 = (1-\lambda)x_1+\lambda c.
\end{equation}
Thus, $(\forall n\in\{2,3,\ldots\})$
$x_n = Ty_{n-1}$ and $y_n=Tx_n$, where
$T \colon x\mapsto (1-\lambda)x + \lambda c$ is $(1-\lambda)$-Lipschitz
continuous with unique fixed point $c$.
This completes the proof of \ref{p:130514ii}.
\end{proof}

\begin{example}[Examples \ref{ex:130430a} and
\ref{ex:130430b} revisited]
\label{ex:130430c}
Consider Proposition~\ref{p:130514} with
$c = -3 < a = 2 < b = 6$.
Then
$\max\{a,b-2a\}=\max\{2,6-2\cdot 2\} = \max\{2,2\} = 2 < 3 = -c
< 4.47 \approx \sqrt{20} = \sqrt{2^2+(6-2)^2}= \sqrt{a^2+(b-a)^2} < 6=b$,
and therefore \eqref{e:130514a} holds.
We have
$b+c = 3$, $2a = 4$,
and $a+c+\sqrt{c^2-a^2} = -1+\sqrt{9-4} = -1+\sqrt{5}\approx 1.23$.
Hence \eqref{e:130514b} turns into
\begin{equation}
\frac{\sqrt{5}-1}{4} < \lambda < \frac{3}{4};
\end{equation}
note that since $(-1+\sqrt{5})/4\approx 0.31$,
we can choose in particular $\lambda=\frac{1}{2}$.
By Proposition~\ref{p:130514}, the MAP with starting point $0$ fails while
the $\frac{1}{2}$-MARP converges linearly with rate $\frac{1}{2}$.
\end{example}

We have seen in the last example that MAP can fail to find a solution
while MARP is able to solve the the problem.
On the other hand, MAP can be faster than MARP:

\begin{example}[MARP and nonsummable relaxation parameters]
\label{ex:1}
Suppose that $X=\RR^2$,
that $A= \RR\times\{0\}$, and that $B=\{0\}\times\RR$.
Then $A\cap B=\{(0,0)\}$.
On the one hand,
regardless of the location of $y_{-1}$, the sequences for the $(1,1)$-MARP,
i.e., MAP, satisfy
$y_0=x_1=y_1=\cdots = 0\in A \cap B$ and thus convergence occurs in
finitely many steps.
On the other hand, let us now consider the MARP.
Writing $y_{-1}=(\eta_1,\eta_2)$, one checks that
for every $\nnn$,
\begin{subequations}
\label{e:130515b}
\begin{equation}
\textstyle x_n=\big(\eta_1\,\prod_{i=0}^{n-1}(1-\mu_i),\eta_2\,
\prod_{i=0}^{n}(1-\lambda_i)\big)
\end{equation}
and
\begin{equation}
\textstyle y_n=\big(\eta_1\,\prod_{i=0}^{n}(1-\mu_i),\eta_2\,
\prod_{i=0}^{n}(1-\lambda_i)\big).
\end{equation}
\end{subequations}
Thus if one of the relaxation parameters encountered is one,
then we obtain finite convergence in the corresponding coordinate.
So assume that $(\forall\nnn)$ $\max\{\lambda_n,\mu_n\}<1$,
that $\eta_1\neq 0$, and that $\eta_2\neq 0$.
If $\lambda_n\to 0$ and $\mu_n\to 0$, then
(similarly to the discussion of Example~\ref{ex:130515a} or
see \cite[Proposition~2.1]{BJMAA}), we have the following
characterizations:
\begin{enumerate}
\item
$(\lim_{n\in\NN}x_{n},\lim_{n\in\NN}y_{n})=(0,0)$
$\Leftrightarrow$
$\sum_\nnn\lambda_n=\sum_\nnn\mu_n=\pinf$.
\item
$(\lim_{n\in\NN}x_{n},\lim_{n\in\NN}y_{n})\in A\smallsetminus\{(0,0)\}$ 
$\Leftrightarrow$
$\sum_\nnn\lambda_n=\pinf, \sum_\nnn\mu_n<\pinf$.
\item
$(\lim_{n\in\NN}x_{n},\lim_{n\in\NN}y_{n})\in B\smallsetminus \{(0,0)\}$ 
$\Leftrightarrow$
$\sum_\nnn\lambda_n<\pinf, \sum_\nnn\mu_n=\pinf$.
\item
$(\lim_{n\in\NN}x_{n},\lim_{n\in\NN}y_{n})\not\in (A\cup B)$ 
$\Leftrightarrow$
$\sum_\nnn\lambda_n<\pinf, \sum_\nnn\mu_n<\pinf$.
\end{enumerate}
This shows that when $\lambda_{\infty}=\mu_{\infty}=0$, 
all possibilities for $\lim_{\nnn}(x_{n},y_{n})$ occur. 
See also Examples~\ref{ex:130515c}, \ref{ex:130515a},
and \ref{ex:130515b}.
\end{example}

\begin{remark}[convergence rates: actual vs upper bounds]
In the previous sections, we have established upper bounds for the linear
convergence rates. Let us now make some comments on the tightness of these
estimates.

Consider the set up in Example~\ref{ex:1}
with $(\forall\nnn)$ $\lambda_n=\mu_n = \lambda\in\left]0,1\right]$.
Then \eqref{e:130515b} yields the \emph{actual rate}
\begin{equation}
\rate_\text{\rm actual}:=1-\lambda\in\left[0,1\right[.
\end{equation}
Let us now turn to the estimates established earlier.
On the one hand, since Corollary~\ref{c:global} holds with $\theta=0$,
we obtain from \eqref{e:rate} that
\begin{equation}
\rate=\sqrt{\lambda^2+(1-\lambda)^2} > 1-\lambda = \rate_\text{\rm actual}.
\end{equation}
On the other hand,
the assumptions of Theorem~\ref{t:main2} are satisfied with
$\ve=0$, $\dd=+\infty$, $S=X=\RR^2$, and $\theta=0$.
Thus, the upper bound computed using \eqref{e:rts} satisfies
\begin{equation}
\rts=1-\lambda=\rate_\text{\rm actual}.
\end{equation}
\end{remark}

\section{A Doubly Non-Superregular Example}

\label{sec:doubly}

In this final section, we assume that $X=\RR^2$.
We shall construct $A$ and $B$ exhibiting various intriguing properties.
We shall use tools from Euclidean geometry.
Given $(s,u,v)\in X^3$, we denote the signed angle from the ray $\RP\times\{0\}$
to $u$ by $\widehat{u}$;
furthermore, $\widehat{u s v}=\widehat{v s u}$ stands for the usual (nonsigned) angle
at the point $s$.

\subsection{The Set Up}

We assume that
\begin{equation}
4w \in \left[0,\tfrac{\pi}{2}\right]
\quad\text{and}\quad
\cos(4w) = \tfrac{3}{4},
\end{equation}
so $w\approx 0.18$. Define
\begin{equation}
f\colon\RR\to\RR\colon
x\mapsto
\begin{cases}0, &\text{if
$x\in\left]-\infty,0\right]\cup\left]1,+\infty\right[$;}\\
(\tan w)\big(x-\tfrac{1}{2^k}\big), &\text{if $x\in\big]\tfrac{3}{2\cdot
2^{k+1}},\tfrac{1}{2^k}\big]$ and $\kkk$;}\\
-(\tan w)\big(x-\tfrac{1}{2^{k+1}}\big),&\text{if
$x\in\big]\tfrac{1}{2^{k+1}},\tfrac{3}{{2\cdot2^{k+1}}}\big]$ and $\kkk$.}
\end{cases}
\end{equation}
Moreover, denote by $\Phi:\RR^2\to\RR^2$ the reflector with respect
to the line $y=(\tan2w)x$.
Now we assume that
(see Figure~\ref{pic:1})
\begin{equation}
A = \menge{(x,y)\in \RR^2}{y\leq f(x)},
\;\;
B = \Phi(A),
\;\;\text{and}\;\;
c = (0,0)\in A \cap B,
\end{equation}
and we also set
\begin{equation}
(\forall\kkk)\quad
s_k:=(\tfrac{1}{2^k},0)
\;\;\text{and}\;\;
z_k:=\Phi(s_k).
\end{equation}

\firstPic

\subsection{The Normal Cones}

\label{ss:ex:last-i}

Note that
\begin{equation}
\label{e:0102NA}
(\forall a\in A)(\forall k\in\{1,2,\ldots\}) \quad N_A(a)\subseteq N_A(s_k):=
\menge{u\in\RR^2}{\tfrac{\pi}{2}-w\leq\widehat{u}\leq\tfrac{\pi}{2}+w}.
\end{equation}
Let $A'$ be the reflection of $A$ about $\RR\times\{0\}$.
Then, since $\widehat{z_0 c s_0}=4w$ and
$B$ is obtained by rotating $A'$ by the angle $4w$ about the origin $c$,
it follows that
\begin{equation}
\label{e:0102NA'}
(\forall a'\in A')\quad N_{A'}(a')\subseteq N_{A'}(s_k)=-N_{A}(s_k)=
\menge{u\in\RR^2}{-\tfrac{\pi}{2}-w\leq\widehat{u}\leq-\tfrac{\pi}{2}+w}
\end{equation}
and
\begin{equation}
\label{e:0102NB}
(\forall b\in B)\quad N_B(b)\subseteq
\menge{u\in\RR^2}{-\tfrac{\pi}{2}+3w\leq\widehat{u}\leq-\tfrac{\pi}{2}+5w}.
\end{equation}

\subsection{The CQ-number at \ $c$ \ associated with \ $(A,\bd A, B,\bd B)$}

\label{ss:ex:last-ii}

Let $\dd>0$.
Then for every $k\in\{1,2,\ldots\}$,
the closed region $W$ (see Figure~\ref{pic:2})
is a subset of $P^{-1}_A(s_k)$;
thus,
\secondPic
\begin{equation}
\pn{A}{\bd B}(s_k)=\cone\big((P^{-1}_A(s_k)\cap\bd B) -s_k\big)
= \cone(W-s_k)=N_A(s_k)
\end{equation}
and
\begin{equation}
\pn{B}{\bd A}(z_k)= N_B(z_k).
\end{equation}
We now compute the CQ-number at $c$ associated with $(A,\bd A,B,\bd B)$ and $\dd$.
Since for every $k\in\{1,2,\ldots\}$, the normal cones
$\pn{A}{\bd B}(s_k)$ and $\pn{B}{\bd A}(z_k)$
are the largest possible,
it suffices in \eqref{e:CQn} to take the supremum over the points
$B(c;\dd)\cap\menge{s_k}{k\in\{1,2,\ldots\}}$ and
$B(c;\dd)\cap\menge{z_k}{k\in\{1,2,\ldots\}}$, respectively:
\begin{equation}
\theta_\dd=\sup\mmenge{\scal{u}{v}}
{\begin{aligned}
&u\in-\pn{A}{\bd{B}}(s_k),v\in \pn{B}{\bd{A}}(z_l),\|u\|\leq 1, \|v\|\leq 1,\\
&\|s_k\|\leq\dd,\|z_l\|\leq\dd, (k,l)\in\{1,2,\ldots\}^2
\end{aligned}}.
\end{equation}

It thus follows from \eqref{e:0102NA'} and \eqref{e:0102NB} that
\begin{equation}
\begin{aligned}
\theta_\dd&=\sup\mmenge{\scal{u}{v}}
{\widehat{u}\in\left[-\tfrac{\pi}{2}-w,-\tfrac{\pi}{2}+w\right]
\;\text{and}\;
\widehat{v}\in\big[-\tfrac{\pi}{2}+3w,-\tfrac{\pi}{2}+5w\big]}\\
&=\cos 2w=\sqrt{\tfrac{1+\cos4w}{2}}=\sqrt{\tfrac{7}{8}}.
\end{aligned}
\end{equation}
Therefore,
\begin{equation}\label{e:0102theta}
(\forall\dd>0)\quad
0.93<\overline{\theta} = \theta_\dd=\sqrt{\tfrac{7}{8}}<0.94.
\end{equation}

\subsection{A lower bound for $\ve$ in the $(\ve,\dd)$-regular case}

\label{ss:ex:last-iii}

Let $k\in\{1,2,\ldots\}$, and
set $d:=\|s_{k+1}-c\|$.
Then $\|z_k-c\|=\|s_k-c\|=2d$.
Now set $\beta_1:=\|z_k-s_{k+1}\|$ and $\beta_2:=\|z_k-s_k\|$.
Noticing that $\widehat{z_k c s_k}=4w$ and using the cosine theorem for the
two triangles $\triangle cz_ks_{k+1}$ and $\triangle cz_ks_k$, we have
\begin{subequations}
\begin{align}
\beta_1^2&=d^2+(2d)^2-2d(2d)(\cos 4w) =d^2+4d^2-4d^2(\tfrac{3}{4})=2d^2,\\
\beta_2^2&=(2d)^2+(2d)^2-2(2d)(2d)(\cos4w) =4d^2+4d^2-8d^2(\tfrac{3}{4})=2d^2.
\end{align}
\end{subequations}
Hence, $\beta_1=\beta_2=d\sqrt{2}$.
This also implies $P_A z_k \subseteq [s,s_k] \cup [s,s_{k+1}]$
(see Figure~\ref{pic:3}).
\thirdPic
The cosine theorem for the triangle $\triangle z_k s_k s_{k+1}$ gives
\begin{equation}
\cos\widehat{s_{k+1}z_ks_k}=\tfrac{\beta_1^2+\beta_2^2-d^2}{2\beta_1\beta_2}
=\tfrac{2d^2+2d^2-d^2}{2d\sqrt{2}d\sqrt{2}}=\tfrac{3}{4}>0.
\end{equation}
So we conclude that $\widehat{s_{k+1}z_ks_k}=4w$.
Next, since $\widehat{ss_{k+1}s_k}=\widehat{s_{k+1}s_k}s=w$, we have
$\widehat{s_{k+1}ss_k}=\pi-2w$. On the other hand,
\begin{equation}
\widehat{ss_kz_k}=w+\widehat{s_{k+1}s_kz_k}=w+\tfrac{\pi-4w}{2}=\tfrac{\pi}{2}-w.
\end{equation}
Altogether,
\begin{equation}
\widehat{s s_k z_k}=\widehat{s s_{k+1} z_k}
=\widehat{s_k s z_k}=\widehat{s_{k+1} s z_k}=\tfrac{\pi}{2}-w,
\end{equation}
i.e., we have two isosceles triangles $\triangle z_kss_{k+1}$ and $\triangle z_kss_k$.
Let $h$ and $h'$ be the two mid-points of $[s,s_k]$ and $[s,s_{k+1}]$.
Then, $P_A z_k=\{h,h'\}$.
Clearly $u:=z_k-h\in\pn{A}{\bd B}(h)$.
Noticing that $\widehat{h' h z_k}=\widehat{shz_k}-\widehat{shh'}=\tfrac{\pi}{2}-w$, we have
\begin{equation}
\scal{\tfrac{u}{\|u\|}}{\tfrac{s_{k+1}-h}{\|s_{k+1}-h\|}}=\cos\widehat{s_{k+1}hz_k}
>\cos\widehat{h' h z_k}=\cos(\tfrac{\pi}{2}-w)
=\sin w> 0.17;
\end{equation}
consequently,
\begin{equation}\label{e:0102a}
\scal{u}{s_{k+1}-h}>(0.17)\cdot\|u\|\cdot\|s_{k+1}-h\|.
\end{equation}
Now we assume that  $A$ is $(\bd B,\ve,\dd)$-regular at $c$
for some $\ve\geq 0$ and $\dd>0$.
Since $s_n\to c$ and $z_n\to c$,
eventually
all the points $s_{k+1}, s_k, z_k, h', h$ lie in $\ball(c;\dd)$.
From the above argument, we have $u\in\pn{A}{\bd B}(h)$
and
\begin{equation}
(0.17)\cdot\|u\|\cdot\|s_{k+1}-h\|<\scal{u}{s_{k+1}-h}
\leq \ve\cdot\|u\|\cdot\|s_{k+1}-h\|.
\end{equation}
Thus
\begin{equation}\label{e:0102ve}
\ve>0.17.
\end{equation}
Similarly, if $B$ is $(\bd A,\ve,\dd)$-regular, then $\ve>0.17$.

\subsection{For the MAP, \cite[Proposition~3.12]{one} is never applicable}

\label{ss:ex:last-iv}

\noindent
Consider \cite[Proposition~3.12]{one} with $(\wt{A},\wt{B})=(\bd A,\bd B)$
and $I=J$ singletons.
Clearly \cite[(51)]{one} holds.
The two assumptions of \cite[Proposition~3.12]{one} are the following:
\begin{enumerate}
\item $A$ is $(\bd B,\ve,3\dd)$-regular.
\item
The CQ-number $\theta_{3\dd}$ at $c$ associated with $(A,\bd A,B,\bd B)$ and
$3\dd$ satisfies $\theta_{3\dd}<1-2\ve$.
\end{enumerate}
Assume that (i) holds.
On the one hand, $\theta_{3\dd}>0.93$ by \eqref{e:0102theta}.
On the other hand,
$\ve>0.17$ by \eqref{e:0102ve}.
If (ii) holds, then we obtain the absurdity
$0.93<\theta_{3\dd}<1-2\ve< 0.66$.
We conclude that (i) and (ii) cannot hold concurrently,
which implies that \cite[Proposition~3.12]{one} is not applicable.

\subsection{For the MAP, \cite[Theorems~3.14 and 3.17]{one} are never applicable}

In view of \eqref{e:0102ve}, we note that
$A$ is not $(\bd B)$-superregular at $c$, and that  $B$ is not
$(\bd A)$-superregular at $c$.
Therefore, the results in \cite{LLM} are not applicable, and neither are
\cite[Theorems~3.14 and 3.17]{one} with $(\wt{A},\wt{B})=(\bd A,\bd B)$
and $I$ and $J$ singletons.

\subsection{For the MARP, we deduce convergence with a linear rate}

Indeed, suppose that $S=X$.
The $(A,X,B,X)$-CQ condition holds, and so does
the $(A,\bd A,B,\bd B)$-CQ condition.
Hence, Theorem~\ref{t:loc-viaCQ} applies and yields local convergence for
the MARP sequences.
Moreover, by \eqref{e:0102theta},
we can make $\ve$ in \eqref{e:ve-to-dd}, and hence $\dd=+\infty$,
arbitrarily large.  Thus the MARP converges with a linear rate regardless
of the starting point. Note that Corollary~\ref{c:global} also
yields the global convergence result.

The figures suggest that the sequences generated by the MAP also converge with a linear rate.
It would be interesting either to find theorems that allow for this
conclusion or to at least obtain a partition of $A$ and $B$ so that
the results of \cite{one} are applicable to
the induced collections $(\wt{A},\wt{B})$.


\small

\section*{Acknowledgments}
HHB was partially supported by a Discovery Grant and an Accelerator
Supplement of the Natural Sciences and Engineering
Research Council of Canada and by the Canada Research Chair Program.
HMP was supported by a PIMS postdoctoral fellowship, University of British Columbia Okanagan internal grant, and University of Victoria internal grant.
XW was partially supported by a Discovery Grant of
the Natural Sciences and Engineering Research Council of Canada.

\end{document}